\documentclass[a4paper,11pt]{article}
\usepackage[margin=1.25in]{geometry}
\usepackage{bm}
\usepackage{mathrsfs}
\usepackage[plain]{algorithm}
\usepackage{algorithmic}
\usepackage{enumitem}
\usepackage{tikz}
\usepackage{booktabs}
\usepackage{amsmath,amssymb,amsthm}
\usepackage{graphicx,epstopdf}
\usepackage[font=small, textfont=it]{caption}
\usepackage[hidelinks]{hyperref}
\usepackage[textsize=footnotesize]{todonotes}
\usepackage{tabularx, multirow}
\usepackage{soul}

\floatstyle{plaintop}
\restylefloat{algorithm}
\usepackage{caption}
\usepackage{cite}

\hypersetup{
    colorlinks=true,
    linkcolor=blue,
    filecolor=blue,      
    urlcolor=blue,
    citecolor=blue,
    pdfpagemode=FullScreen,
    }

\newcommand*\samethanks[1][\value{footnote}]{\footnotemark[#1]}

\title{Fast Ewald Summation using Prolate Spheroidal Wave Functions}
\author{
    Erik Bostr\"om\thanks{Division of Mathematics and Physics, M\"alardalen University, 72220 V\"aster\aa s, Sweden (\em \{erik.bostrom, ludvig.af.klinteberg\}@mdu.se)},
    Anna-Karin Tornberg\thanks{Department of Mathematics, KTH Royal Institute of Technology, 10044 Stockholm, Sweden (\em akto@kth.se)},
    Ludvig af Klinteberg\samethanks[1]
    }
\date{} 

\graphicspath{{figures}}

\newcommand{\R}{\mathbb{R}}
\newcommand{\Z}{\mathbb{Z}}

\newcommand{\bx}{\bm{x}}
\newcommand{\bxi}{\bm{\xi}}
\newcommand{\by}{\bm{y}}
\newcommand{\bk}{\bm{k}}
\newcommand{\br}{\bm{r}}
\newcommand{\bwv}{\bm{\omega}}
\newcommand{\bomega}{\bm{\omega}}

\newcommand{\bkint}{\bm{\kappa}}

\renewcommand{\d}{\mathrm{d}}

\newcommand{\Kmax}{K_{\text{max}}}

\renewcommand{\i}{\mathrm{i}}
\newcommand{\ai}{a}
\newcommand{\aii}{b}
\renewcommand{\O}{\mathcal{O}}

\DeclareMathOperator{\erf}{erf}
\DeclareMathOperator{\erfc}{erfc}

\DeclareMathOperator{\spt}{supp}

\newtheorem{lemma}{Lemma}
\newtheorem{theorem}{Theorem}

\newtheorem{proposition}{Proposition}
\newtheorem{remark}{Remark}

\theoremstyle{remark}
\theoremstyle{definition}
\newtheorem{alg}{Algorithm}

\newtheorem{example}{Example}

\newcolumntype{C}{>{\centering\arraybackslash}X}

\begin{document}
\maketitle
\begin{abstract}
Fast Ewald summation efficiently evaluates Coulomb interactions and is widely used in molecular dynamics simulations.  It is based on a split into a short-range and a long-range part, where evaluation of the latter is accelerated using the fast Fourier transform (FFT). The accuracy and computational cost depend critically on the mollifier in the kernel split and the window function used in the spreading and interpolation steps that enable the use of the FFT.  The first prolate spheroidal wavefunction (PSWF) has optimal concentration in real and Fourier space simultaneously, and is used when defining both a mollifier and a window function.  We provide a complete description of the method and derive rigorous error estimates.  In addition, we obtain closed-form approximations of the Fourier truncation and aliasing errors, yielding explicit parameter choices for the achieved error to closely match the prescribed tolerance.  Numerical experiments confirm the analysis: PSWF-based Ewald summation achieves a given accuracy with significantly fewer Fourier modes and smaller window supports than Gaussian- and B-spline-based approaches, providing a superior alternative to existing Ewald methods for particle simulations.
\end{abstract}

\medskip


\bigskip


\section{Introduction}
\label{sec:intro}

Long-range interactions such as Coulomb and Stokes potentials arising in particle-based systems are central to molecular dynamics (MD), fluid dynamics, and wave propagation. A major computational challenge is that direct evaluation of such interactions scales quadratically with the number of particles. Ewald summation, introduced in 1921~\cite{ewaldBerechnungOptischerUnd1921} reduces this cost and has since become a standard tool in large-scale simulations.

In Ewald's original formulation, the three-dimensional Laplace kernel $G(r)=1/r$ ($r=|\bx|$, $\bx\in\R^3$) is split into a mollified kernel $M(r)$ and a residual kernel $R(r)$ according to
\begin{align}\label{eq:split}
    \frac{1}{r} = M(r) + R(r) := \frac{\erf(r/\sigma)}{r} + \frac{\erfc(r/\sigma)}{r},
\end{align}
where $\erf(x) = \tfrac{2}{\sqrt{\pi}}\int_0^x e^{-t^2}\,\mathrm{d}t$ and $\erfc(x) = 1-\erf(x)$ are the error and complementary error functions. The parameter $\sigma>0$ is a width-parameter of the mollification that controls the decay of both terms. In Ewald's original kernel split \eqref{eq:split}, the mollified kernel $M$ is the convolution of $G$ with a Gaussian $\tfrac{1}{\sqrt{\pi}\sigma} e^{-r^2/\sigma^2}$, while the residual is $R=G-M$. Because a Gaussian is not compactly supported, $R$ must be truncated numerically outside a cutoff radius $r_c>0$, with truncation error controlled by a numerical tolerance that is typically denoted by $\varepsilon$.

The Ewald split enables efficient evaluation of the potential $\phi\colon\R^3\to\R$ at the particle locations $\{\bx_j\}_{j=1}^n$ in a triply periodic domain $\Omega=[0,L)^3$ (cubic for clarity of presentation; the method extends directly to general periodic lattices). Using \eqref{eq:split}, it can be written as
\begin{align}
    \phi(\bx_i) & = \sum_{\bm{r}\in\Z^3}\sum_{j=1}^n {\vphantom{\sum}}' G(|\bx_i-\bx_j+L\bm{r}|)\,\rho_j,\label{eq:potential1}\\  
    \begin{split}
        & = \frac{1}{L^3}\sum_{\bk\neq \bm{0}} \hat{M}(\bk)\,\sum_{j=1}^n \rho_j e^{-\frac{2\pi}{L} i \bk\cdot(\bx_i-\bx_j)}
          - \lim_{r\to 0} M(r)\rho_i \\
        & \quad + \sum_{\bm{r}\in\Z^3}\sum_{j=1}^n {\vphantom{\sum}}' R(|\bx_i-\bx_j+L\bm{r}|)\,\rho_j,
        \qquad i=1,\ldots,n,
    \end{split}\label{eq:potential2}
\end{align}
where $\{\rho_j\}_{j=1}^n$ are particle strengths (e.g., charges), $\bm{r}\in\Z^3$ indexes the periodic images, and the prime indicates that the self-interaction ($i=j$, $\bm{r}=0$) is omitted. The Ewald split \eqref{eq:split} makes the otherwise conditionally convergent sum \eqref{eq:potential1} convergent provided that the charge-neutrality condition $\sum_{j=1}^n\rho_j = 0$ holds. In~\eqref{eq:potential2}, the mollified part $M$ is evaluated in Fourier space, where $\hat{M}(\bk)$ decays rapidly, while the residual $R$ captures local interactions within a cutoff radius $r_c$. This reduces the computational cost of evaluating $\phi$ from $\mathcal{O}(n^2)$ to $\mathcal{O}(n^{3/2})$ for close to randomly uniform particle distributions~\cite{hockneyComputerSimulationUsing2021}.

Several fast Ewald variants accelerate the Fourier-space sum further to $\mathcal{O}(n\log n)$ using the fast Fourier transform (FFT) \cite{cooleyAlgorithmMachineCalculation1965}. These include the Particle Mesh Ewald (PME)~\cite{dardenParticleMeshEwald1993}, Smooth Particle Mesh Ewald (SPME)~\cite{essmannSmoothParticleMesh1995}, Particle--Particle--Particle--Mesh Ewald (PPPM or P$^3$M)~\cite{hockneyComputerSimulationUsing2021}, the Spectral Ewald (SE)~\cite{lindboSpectralAccuracyFast2011,shamshirgarFastEwaldSummation2021,afklintebergFastEwaldSummation2017}, and the Particle--Particle NFFT Ewald (P$^2$NFFT) method~\cite{nestlerFastEwaldSummation2015,hofmannNFFTBasedEwald2017a}. These schemes spread particle data onto a uniform grid using a localized window function, apply the FFT, scale in Fourier space, and interpolate back. The accuracy and efficiency therefore depend on two distinct choices of functions:  
(i) the mollifier used in the kernel split, and  
(ii) the window function used in spreading and interpolation.  
Both influence the spectral decay and thus the number of Fourier modes and grid points required for a target accuracy. Classical methods differ mainly in these choices: PME, SPME, and P$^3$M employ low-order B-splines as window functions, while SE and P$^2$NFFT use Gaussians or various types of approximations to the first prolate spheroidal wavefunction of order zero (PSWF).

The PSWF, studied extensively by Slepian and collaborators~\cite{slepianProlateSpheroidalWave1961,slepianCommentsFourierAnalysis1983} and in \cite{osipovProlateSpheroidalWave2013,osipovEvaluationProlateSpheroidal2014}, is optimally concentrated in both real and Fourier space. This makes it attractive both as a mollifier in the kernel split and as a window function for spreading/interpolation, though in principle different PSWFs with different parameters may be used for the two roles. Historically, the absence of a closed form limited practical use, and approximations such as the Kaiser--Bessel (KB) and exponential of semicircle (ES) functions were employed instead~\cite{barnettParallelNonuniformFast2019a,keiner2009using}, including in SE and P$^2$NFFT implementations~\cite{shamshirgarFastEwaldSummation2021,nestlerFastEwaldSummation2015}. With modern evaluation algorithms~\cite{osipovEvaluationProlateSpheroidal2014} and polynomial approximations~\cite{barnettParallelNonuniformFast2019a}, PSWFs have become computationally practical.

Until recently, PSWF analogues, such as the KB and ES functions, have primarily been considered as window functions in fast Ewald summation. However, employing them also as mollifiers in the kernel split yields further advantages, reducing the number of required Fourier modes by around a factor of two in each dimension compared to Gaussians, as illustrated in Figure~\ref{fig:PSWFvsGaussian}.
PSWF-based kernel splitting was first introduced for the Laplace and Yukawa kernels in the Dual-space Multilevel Kernel-splitting framework (DMK)~\cite{jiangDualspaceMultilevelKernelsplitting}, and was later extended to kernels arising in Stokes flow~\cite{klintebergFastSummationStokes2025}.
\begin{figure}[t!]
\centering
\includegraphics[width=0.95\textwidth]{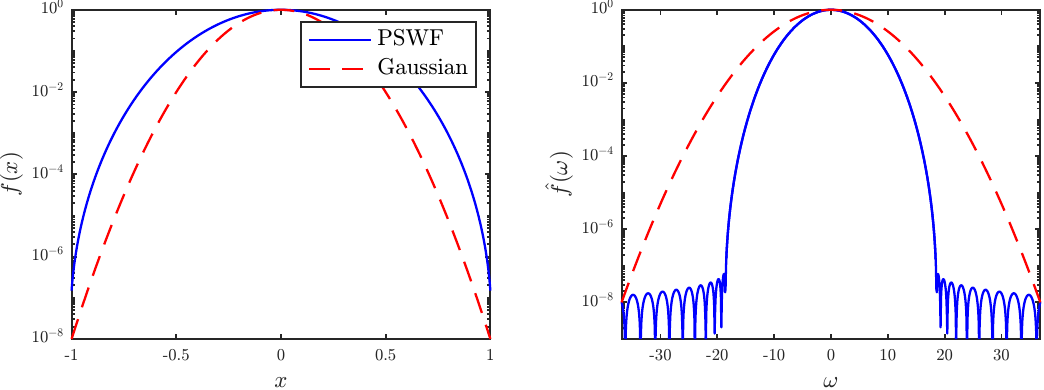}
\caption{Comparison of the Gaussian and the first PSWF of order zero in real space (left) 
and Fourier space (right). For equal effective support, the PSWF mollifier requires roughly 
half as many Fourier modes to achieve the same accuracy.}
\label{fig:PSWFvsGaussian}
\end{figure}
Fast Ewald summation based on this type of kernel splitting 
was developed independently in the present work and by Liang et al.~\cite{liangAcceleratingFastEwald2025}. In their work, a PSWF-based fast Ewald summation method for Coulomb interactions in molecular dynamics is introduced, in which PSWFs are used both for the kernel split and for spreading/interpolation, and the resulting method is integrated into the MD codes GROMACS~\cite{vanderspoelGROMACSFastFlexible2005}
and LAMMPS~\cite{LAMMPS}. Their benchmarks show substantial practical speedups, with total computation times reduced by up to a factor of three when running large-scale simulations with thousands of cores.

The present work is complementary to \cite{liangAcceleratingFastEwald2025}. Our focus is on the mathematical formulation and analysis of PSWF-based Ewald summation, including:
\begin{itemize}
    \item Explicit definitions of PSWF-based mollifiers and PSWF-based window functions, integrated consistently within the Ewald method.
    \item Rigorous error estimates for both splitting and windowing, with simplified closed-form approximations for practical use.
    \item Guidelines for parameter selection for the Fourier sum, enabling predictive accuracy control without ad hoc tuning.
    \item Systematic comparisons with Gaussian- and B-spline-based methods, demonstrating that PSWF-based approaches achieve equal accuracy with fewer Fourier modes and smaller window supports.
\end{itemize}
The resulting method is optimal with respect to computational complexity, due to the minimal support of the PSWF. While our analysis focuses on the triply periodic electrostatic potential, it extends naturally to other periodicities and kernels, including Stokes flow~\cite{baggeFastEwaldSummation2023,afklintebergFastEwaldSummation2017}.


The remainder of the paper is organized as follows. Section~\ref{sec:matprel} introduces mathematical preliminaries. Section~\ref{sec:fastewald} reviews the direct and fast Ewald summation algorithms. Section~\ref{sec:fast_Ewald_PSWF} presents the PSWF-based method. Section~\ref{sec:error_analysis} provides rigorous error analysis. Section~\ref{sec:parameter_selection} discusses parameter selection. Section~\ref{sec:num_results} reports numerical experiments, and Section~\ref{sec:conclusions} concludes the paper.

\section{Mathematical preliminaries}
\label{sec:matprel}
This section introduces notation and the fundamental concepts used throughout the paper.

\subsection{Basic definitions and notation}
\label{sec:def}

We consider the periodic box $\Omega=\R^3/(L\Z^3)\simeq[0,L)^3$ with side length $L$ and volume $V=L^3$. Particle locations are $\{\bx_j\}_{j=1}^n\subset\Omega$ with associated strengths $\{\rho_j\}_{j=1}^n$. We write $r=|\bx|$. 

Fourier-space variables are $\bk\in\Z^3$ and $\bomega:=(2\pi/L)\bk\in\R^3$ with $\omega=|\bomega|$. We use the Fourier transform convention $\hat f(\bomega)=\int_{\R^3} f(\bx)\,e^{-\i\bomega\cdot\bx}\,d\bx$. When convenient, we write $\hat f(\bk):=\hat f\!\left((2\pi/L)\bk\right)$. 
Fourier-series coefficients are $c_{\bk}(f)$, and discrete Fourier transform coefficients are denoted $\hat f_{\bk}\approx c_{\bk}(f)$. Radially symmetric functions are written $f(r)$ with $r=|\bx|$.

For a cubic bandwidth $\bm m=(m,m,m)$ define
\begin{align}
\mathcal{I}_m & :=
\begin{cases}
\{-m/2,\ldots,m/2-1\}, & m \text{ even},\\
\{-(m-1)/2,\ldots,(m-1)/2\}, & m \text{ odd},
\end{cases}
\qquad
\mathcal{I}_{\bm m} := \mathcal{I}_m^3,\qquad |\mathcal{I}_{\bm m}|=m^3.
\end{align}
Define $K_{\max}:=\lfloor m/2\rfloor$ and $\omega_{\max}:=(2\pi/L)K_{\max}$. 

\subsection{A fundamental lemma for Laplacian kernel splitting}

The next identity links a one-dimensional function to its three-dimensional
convolution with the Laplace kernel $G(\bx)=1/|\bx|$.

\begin{lemma}\label{lem:laplacian}
Let $f\in L^1(\R)$ be even and non-negative, with Fourier transform $\hat f$.
Define the radial function $g\colon\R^3\to\R$ by
\begin{align}
g(\bx) := \frac{2}{|\bx|}\int_0^{|\bx|} f(u)\,du,\qquad \bx\neq \bm{0},
\end{align}
with $g(\bm{0})$ defined by continuity. Then, with Fourier transform taken in $\R^3$,
\begin{align}
\hat g(\bomega) \;=\; \frac{4\pi}{|\bomega|^2}\,\hat f(|\bomega|),\qquad \bomega\neq \bm{0}.
\end{align}
\end{lemma}

\begin{remark}[Radial extension]
Let $\eta\colon\R^3\to\R$ be the radial extension of $f$ from Lemma~\ref{lem:laplacian}, defined by $\eta(\bx):=f(|\bx|)$. Then $\hat\eta(\bomega)=\hat f(|\bomega|)$, and thus $\hat g(\bomega)=\hat G(\bomega)\,\hat\eta(\bomega)$ with $\hat G(\bomega)=4\pi/|\bomega|^2$, i.e.\ $g=G*\eta$. See Appendix~\ref{app:radial-extension} for details.
\end{remark}

A proof of Lemma \ref{lem:laplacian} using spherical coordinates and integration by parts is given in Appendix A.3 of \cite{jiangDualspaceMultilevelKernelsplitting}.

\subsection{Numerical mollifiers for the Laplacian kernel}
\label{sec:num_mollifiers}

We call an even and non-negative function $\gamma_\sigma \in L^1(\R)$ with width 
parameter $\sigma>0$ a \emph{numerical mollifier} (or simply a \emph{mollifier}) if
\begin{align}
    \int_{\R} \gamma_{\sigma}(x)\,dx = 1,
\end{align}
and
\begin{align}
    |\gamma_\sigma(x)| \le \delta \qquad \text{for all } |x| \ge \sigma,
\end{align}
where $\delta>0$ is a prescribed tolerance controlling the tail of $\gamma_\sigma$.

By Lemma~\ref{lem:laplacian}, the mollified Laplace kernel and its Fourier transform are
\begin{align}\label{eq:mollified_kernel0}
    M(\bx) := (G \ast \eta_{\sigma})(\bx) 
    = \frac{2\int_0^{|\bx|} \gamma_{\sigma}(u)\,du}{|\bx|},
\end{align}
\begin{align}\label{eq:mollified_kernel}    
    \hat{M}(\bomega) = \hat{G}(\bomega)\,\hat{\eta}_{\sigma}(\bomega) 
    = \frac{4\pi}{|\bomega|^2} \, \hat{\gamma}_{\sigma}(|\bomega|).
\end{align}
Here $\eta_{\sigma}$ is the radial extension of $\gamma_{\sigma}$. In Ewald summation, 
\(M\) is the \emph{mollified kernel} and \(R := G - M\) the \emph{residual kernel}, with 
$\sigma$ chosen so that $|R(r_c)| \lesssim \varepsilon$ at the cutoff radius $r_c$.

\begin{remark}
In previous work, e.g.\ \cite{lindboSpectralAccuracyFast2011}, the residual kernel $R$ is not
compactly supported, and the parameters $\sigma$ and $k_{\max}$ are chosen so that the
resulting error matches approximate error models, rather than by imposing a condition
such as $|R(r_c)| \lesssim \varepsilon$. The condition $|R(r_c)| \lesssim \varepsilon$ nevertheless
provides a sufficient starting point for subsequent fine-tuning of the parameters; see
also \cite{liangAcceleratingFastEwald2025}.
 Due to the construction of the mollifier, we can
instead achieve $R(r_c)=0$ (as shown in Section~\ref{sec:fast_Ewald_PSWF}), which
eliminates truncation errors in the residual kernel and makes the parameter selection
process easier and entirely determined by the Fourier-space sum.
\end{remark}

\begin{example}[Gaussian mollifier]\label{ex:Gaussian_split}
The 1D Gaussian
\begin{align}
    g(x) = \frac{1}{\sqrt{\pi}}\,e^{-x^2},
    \qquad 
    \hat g(\omega) = e^{-\omega^2/4},
\end{align}
scaled by a width parameter $\sigma > 0$,
\begin{align}
  g_\sigma(x) = \frac{1}{\sqrt{\pi}\sigma}\,e^{-x^2/\sigma^2},
  \qquad
  \hat g_\sigma(\omega) = e^{-\sigma^2\omega^2/4},
\end{align}
is a mollifier in the sense of the definition above.
It yields the classical Gaussian Ewald split
\begin{align}
    M(\bx) = \frac{\erf(|\bx|/\sigma)}{|\bx|}, \qquad 
    \hat{M}(\bomega) = \frac{4\pi}{|\bomega|^2}\,e^{-\sigma^2 |\bomega|^2/4},
\end{align}
with residual
\begin{align}
  R(|\bx|) = \frac{\erfc(|\bx|/\sigma)}{|\bx|}.
\end{align}
Since $\erfc(|\bx|/\sigma)$ decays exponentially, the residual kernel becomes smaller than $\varepsilon$ for $|\bx| \gtrsim r_c$, provided that
\begin{align}
  (r_c/\sigma)^2 \gtrsim \log(1/\varepsilon).
\end{align}
\end{example}

\subsection{Window functions, periodization and Fourier coefficients}
\label{sec:window_and_periodization}

We define a \emph{window function} as a real and even function $\varphi \in L^2(\R^3)$ for which both $\varphi$ and $\hat\varphi$ are strongly localized near zero, enabling compact truncation in real space and rapid spectral decay.

If a window function without compact support is used, it can be truncated. We define its truncation by
\begin{align}
    \varphi_t(\bx) := 
    \begin{cases}
        \varphi(\bx), & \bx \in [-\alpha,\alpha]^3,\\
        0, & \bx \notin [-\alpha,\alpha]^3.
    \end{cases}
\end{align}

On the periodic domain $\Omega$, a periodized window function $\tilde\varphi$ is employed. More generally, for any function $f:\R^3 \to \R$ we define the periodization $\tilde f$ on $\Omega$ by
\begin{align}
    \tilde f(\bx) := \sum_{\br \in \Z^3} f(\bx + L\br),
\end{align}
which admits the Fourier series
\begin{align}
    \tilde f(\bx)
    = \sum_{\bk \in \Z^3} c_{\bk}(\tilde f)\, 
      e^{\frac{2\pi}{L}\, \i\, \bk \cdot \bx},
\end{align}
with coefficients (see Appendix~\ref{sec:fcoeffs} for the derivation)
\begin{align}
    c_{\bk}(\tilde f)
    = \frac{1}{V}\int_{\Omega} 
        \tilde f(\bx)\, e^{-\frac{2\pi}{L}\, \i\, \bk \cdot \bx}\, d\bx
    \;=\; \frac{1}{V}\, \hat f\!\left(\tfrac{2\pi}{L}\bk\right).
    \label{eq:fcoeffs}
\end{align}
Hence
\begin{align}
    \tilde f(\bx)
    = \frac{1}{V}\sum_{\bk \in \Z^3}
      \hat f\!\left(\tfrac{2\pi}{L}\bk\right)\,
      e^{\frac{2\pi}{L}\, \i\, \bk \cdot \bx}.
\end{align}

\begin{remark}
We distinguish Fourier series coefficients from discrete Fourier transform (DFT) coefficients.
For $f\in L^2(\Omega)$ sampled on an $m^3$ uniform grid, the DFT produces $\hat f_{\bk}$ for $\bk\in\mathcal I_{\bm m}$, related to $c_{\bk}(f)$ by the aliasing identity \cite{plonkaNumericalFourierAnalysis2018}
\begin{align}\label{eq:alias}
    \hat f_{\bk} \;=\; \sum_{\br\in\Z^3} c_{\bk + m\br}(f).
\end{align}
\end{remark}

\subsection{Prolate Spheroidal Wave Functions}

The zeroth order\footnote{Here ``order'' refers to the angular (azimuthal) index in the
prolate spheroidal wave functions; order zero corresponds to the
axisymmetric family.} \emph{prolate spheroidal wave functions} (PSWFs), $\{\psi_j^c\}_{j=0}^\infty$, arise as eigenfunctions of the integral operator
\begin{align}\label{eq:operator}
    F_c[f](s) := \int_{-1}^1 f(t)\, e^{\i c s t}\,dt, \qquad s\in[-1,1],
\end{align}
where $c>0$ is the \emph{bandlimit}. They satisfy
\begin{align}\label{eq:eigenvalueproblem}
    \lambda_j \psi_j^c(s) = \int_{-1}^1 \psi_j^c(t)\, e^{\i c s t}\,dt,\qquad s\in\R,
\end{align}
with eigenvalues ordered by decreasing magnitude
\begin{align}
\sqrt{\frac{2\pi}{c}} > |\lambda_0(c)| > |\lambda_1(c)| > \cdots > 0.
\end{align}
The eigenfunctions $\psi_j^c$ are taken real and orthonormal in $L^2([-1,1])$, 
so that $\|\psi_j^c\|_{L^2([-1,1])}=1$. They alternate parity with $j$ (even for even $j$, odd for odd $j$) and are 
real–analytic on $(-1,1)$; in particular $\psi_j^c\in C^\infty([-1,1])$ 
\cite{osipovProlateSpheroidalWave2013}. 
For definiteness, we fix the sign of $\psi_0^c$ by requiring 
$\psi_0^c(0)>0$.

Among bandlimited functions, the first PSWF $\psi_0^c$ ($j=0$) maximizes spectral concentration:

\begin{theorem}[see {\cite[Theorem~3.53]{osipovProlateSpheroidalWave2013}}]\label{thm:bandlimited}
Let $f\in L^2([-1,1])$ with $\|f\|_{L^2([-1,1])}=1$ and $c>0$. Define $g$ by
\begin{align}\label{eq:bandlimited}
    g(s) = \int_{-1}^1 f(t)\, e^{\i c s t}\,dt, \qquad s\in\R.
\end{align}
Then
\begin{align}
    \|g\|^2_{L^2(\R)} = \frac{2\pi}{c}, 
    \qquad 
    \|g\|^2_{L^2([-1,1])} \le |\lambda_0(c)|^2,
\end{align}
with equality for $f=\psi_0^c$.
\end{theorem}

Thus, in the present work we only need $\psi_0^c$, which is even has no zeros in $(-1,1)$ and, 
with the normalization $\psi_0^c(0)>0$, satisfies $\psi_0^c(s)>0$ for all 
$s\in(-1,1)$; in particular, $\psi_0^c(1)>0$ \cite{osipovProlateSpheroidalWave2013}.

The Fourier transform of the truncation of $\psi_0^c$ to the interval $[-1,1]$ satisfies a simple relation:

\begin{lemma}\label{lem:PSWFrelation}
Let $f:\R\to\R$ be defined by
\begin{align}
f(s)=
\begin{cases}
\psi_0^{c}(s), & s\in[-1,1],\\
0, & \text{otherwise}.
\end{cases}
\end{align}
Then
\begin{align}
\hat f(\omega)=\lambda_0\,\psi_0^{c} (\omega/c),
\qquad \omega\in\mathbb{R}.
\end{align}
In particular,
\begin{align}
\hat f(\omega)=\lambda_0\,f(\omega/c),\qquad \omega\in[-c,c].
\end{align}
\end{lemma}
\begin{proof}
The Lemma follows directly from \eqref{eq:eigenvalueproblem} letting $\omega=cs$.
\end{proof}

For computation, $\psi_0^c$ in $[-1,1]$ can be obtained from its Legendre expansion and a tridiagonal eigenproblem derived from the Sturm--Liouville equation \cite{osipovEvaluationProlateSpheroidal2014}
\begin{align}\label{eq:ode}
    (1 - s^2) \frac{\mathrm{d}^2 \psi_0^c}{\mathrm{d}s^2} - 2s\frac{\d\psi_0^c}{\d s} + (\chi_0 - c^2 s^2)\psi_0^c = 0,
\end{align}
where $\chi_0>0$ is the differential-equation eigenvalue associated with $\psi^c_0$. Details of the computational problem are given in Appendix~\ref{app:pswf}.

The differential equation \eqref{eq:ode} is also useful for establishing several properties of the PSWFs. In order to do so, the following relation between the first eigenvalue $\chi_0$ and the bandlimit $c$ is useful. 

\begin{lemma}\label{lem:chi0}
Let $c>0$. Then the smallest eigenvalue associated with \eqref{eq:ode} satisfies
\begin{align}\label{eq:ineq1}
\chi_0 \le \frac{c^2}{3}.
\end{align}
\end{lemma}
From \eqref{eq:ineq1}, it follows that for $s>1$, $
c^2 s^2 - \chi_0 \ge 2c^2/3 > 0$. This yields the following tail bound, which will be used in the error estimate in Proposition~\ref{prop:PSWF_split_errors} of Section \ref{sec:dirsumerrors}.
\begin{lemma}\label{lem:tailbnd}
Let $c \ge \sqrt{3}$ and let $s\ge 1$. Then
\begin{align}\label{eq:bnd}
|\psi_0^c(s)| \le \frac{\psi_0^c(1)}{s}.
\end{align}
\end{lemma}
Proofs of Lemmas \ref{lem:chi0} and \ref{lem:tailbnd} are provided in Appendix~\ref{sec:tailbnd}.

For $|s|>1$, $\psi_0^c$ exhibits oscillatory behavior, as described in the following lemma.
\begin{lemma}[see {\cite[Theorem~3.44]{osipovProlateSpheroidalWave2013}}]\label{lem:exp}
Suppose that $c>0$ and that $s>1$. Suppose furthermore that $A^c(s)$ is defined via the formula
\begin{align}
A^c(s) = \frac{1}{\lambda_0 \psi_0^c(1)}\int_{-1}^1\frac{\sin(c(s-t))\psi_0^c(t)t}{s-t}\,\d t.
\end{align}
Then,
\begin{align}\label{eq:sin}
\psi_0^c(s) = \frac{2\psi_0^c(1)}{c s \lambda_0} \left[\sin(cs) + A^c(s)\right].
\end{align}
\end{lemma}
The representation \eqref{eq:sin} implies the approximation 
$\psi_0^c(s) \approx \frac{2\psi_0^c(1)}{\lambda_0}\frac{\sin(cs)}{cs}$, 
and the estimate 
$\psi_0^c(s) \lesssim \frac{2\psi_0^c(1)}{\lambda_0}\frac{1}{cs}$ for $s\ge 1$. 
These approximations are supported by numerical observations (see Figure~\ref{fig:pswfbnd}) 
and are used in the error model in Section~\ref{sec:spliterror}.
\begin{figure}[tb!]
\includegraphics[width=0.295\textwidth]{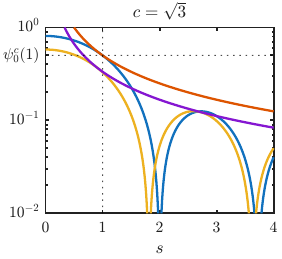}
\includegraphics[width=0.705\textwidth]{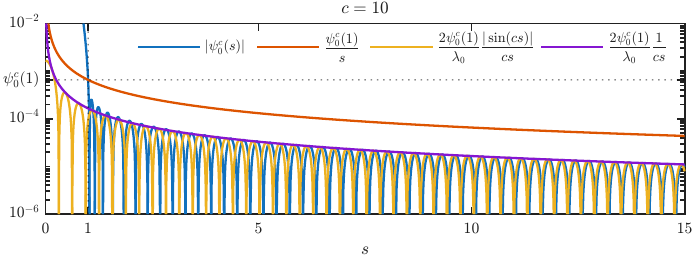}
\caption{The tail of $|\psi_0^c(s)|$ for $c=\sqrt{3}$ and $c=10$, shown together with the bound \eqref{eq:bnd}, the approximation obtained from \eqref{eq:sin} by neglecting $A^c(s)$, and its corresponding upper bound. By Lemma~\ref{lem:tailbnd}, the bound $|\psi_0^c(s)| \le \psi_0^c(1)/s$ holds for $c\ge\sqrt{3}$ (left). For large $s$, the approximation $\tfrac{2\psi_0^c(1)}{\lambda_0}\,\tfrac{|\sin(cs)|}{cs}$ accurately captures the oscillatory tail (right).}\label{fig:pswfbnd}
\end{figure}

\section{Ewald summation}
\label{sec:fastewald}

Before we consider PSWF-based implementations, we briefly review the classical Ewald summation and its components. We then summarize the general fast Ewald algorithm, in which the PSWF will serve as both mollifier and window.

\subsection{Ewald summation for triply periodic problems}

Recall that the kernel split of a radial kernel $G\colon \mathbb{R}^3\to \mathbb{R}$ is
\begin{align}
    G(r) = M(r) + R(r),
\end{align}
where $M$ and $R$ are the mollified and residual kernels, respectively. Periodizing this decomposition on $\Omega$ gives
\begin{align}\label{eq:Ewald}
    \sum_{\bm{r}\in\mathbb{Z}^3} G(|\bm{x}+L\bm{r}|) 
    = \sum_{\bm{r}\in\mathbb{Z}^3} M(|\bm{x}+L\bm{r}|) 
    + \sum_{\bm{r}\in\mathbb{Z}^3} R(|\bm{x}+L\bm{r}|).
\end{align}
Using this, the \emph{Ewald sum} for the triply periodic potential $\phi$ evaluated at particle locations $\bm{x}_i \in \Omega$ ($i=1,\ldots,n$) is
\begin{align}
    \phi_i := \phi(\bm{x}_i)= \sum_{\bm{r}\in\mathbb{Z}^3} \sum_{j=1}^n{\vphantom{\sum}}' 
    G(|\bm{x}_i - \bm{x}_j + L\bm{r}|)\, \rho_j 
    \;=\; \phi_i^{\text{local}} + \phi_i^{\text{far}} + \phi_i^{\text{self}},
\end{align}
where
\begin{align}
    \phi_i^{\text{local}} &:= \sum_{\bm{r}\in\mathbb{Z}^3} \sum_{j=1}^n{\vphantom{\sum}}' R(|\bm{x}_i - \bm{x}_j + L\bm{r}|)\, \rho_j, \\
    \phi_i^{\text{far}} &:= \sum_{\bm{r}\in\mathbb{Z}^3} \sum_{j=1}^n M(|\bm{x}_i - \bm{x}_j + L\bm{r}|)\, \rho_j, \label{eq:farfield}\\
    \phi_i^{\text{self}} &:= -\lim_{|\bm{x}|\to 0} M(|\bm{x}|)\, \rho_i.
\end{align}
The self-term removes the self-interaction at $\bm{x}_i=\bm{x}_j$ that would otherwise be included in $\phi_i^{\text{far}}$. 

For $G(r)=1/r$, the residual kernel $R$ decays rapidly in real space. Moreover,
$\hat M(\bm{\omega})=4\pi\,\hat\gamma_\sigma(|\bm{\omega}|)/|\bm{\omega}|^2$, and if $\gamma_\sigma$ is compactly supported and sufficiently smooth, then by the Paley–Wiener theorem $\hat M$ decays superalgebraically, enabling efficient evaluation of $\phi_i$ \cite{plonkaNumericalFourierAnalysis2018}.

\begin{remark}
Given that $\phi_i$ denotes the electrostatic potential at $\bm{x}_i$, the total electrostatic energy is
\begin{align} 
E \;=\; \frac{1}{2} \sum_{i=1}^n \rho_i\, \phi_i
\;=\; \frac{1}{2} \sum_{\bm{r}\in\mathbb{Z}^3} \sum_{i=1}^n \sum_{j=1}^n{\vphantom{\sum}}' 
\rho_i \rho_j\, G\!\left(|\bm{x}_i - \bm{x}_j + L\bm{r}|\right),
\end{align}
and the force acting on particle $i$ is given by
\begin{align}
\bm{F}_i \;=\; -\rho_i \nabla \phi_i
\;=\; - \sum_{\bm{r}\in\mathbb{Z}^3} \sum_{j=1}^n{\vphantom{\sum}}' 
\rho_i \rho_j\, \nabla G\!\left(|\bm{x}_i - \bm{x}_j + L\bm{r}|\right).
\end{align}
The practical evaluation of the force within the proposed Ewald method, including the corresponding real- and Fourier-space contributions, is discussed in detail in Remark~\ref{rem:force} (see in particular \eqref{eq:force}).
\end{remark}

\subsubsection{Computing the real-space sum}\label{sec:realspace}

The real-space sum $\phi_i^{\text{local}}$ is computationally efficient because the kernel $R$ decays rapidly. If $R$ is not compactly supported, it is truncated at a cutoff radius $r_c$ such that $|R(r_c)| \lesssim \varepsilon$, ensuring that the truncation error is of the same order of magnitude as the prescribed tolerance (cf.~Example~\ref{ex:Gaussian_split}). Consequently, only particles within a distance $r_c$ of $\bx_i$ are included in the computation. For uniform particle distributions, let $n_{r_c}$ denote the number of particles within distance $r_c$ of each $\bx_i$. Then, when, e.g., neighbor or cell lists are employed \cite{hockneyComputerSimulationUsing2021}, the total computational cost of this evaluation scales as $\mathcal{O}(n_{r_c} n)$.

\subsubsection{Direct computation of the Fourier-space sum}
\label{sec:Fourier_sum}

Using \eqref{eq:fcoeffs}, the far-field component \eqref{eq:farfield} is approximated by the truncated Fourier representation
\begin{align}\label{eq:Fsum}
    \phi_i^{\text{far}}
    \approx
    \phi_t^{\text{far}}(\bx_i) := 
    \frac{1}{V}\sum_{\bm{k}\in\mathcal{I}_{\bm{m}}\setminus\{\bm{0}\}} 
        \hat{M}(\bm{k})\,\hat{\rho}(\bm{k})\,
        e^{-\frac{2\pi}{L}\i\,\bm{k}\cdot \bm{x}_i},
    \qquad i=1,\ldots,n.
\end{align}
Here, for the Laplace kernel $\hat{G}(\bomega) = 4\pi/|\bomega|^2$,
\begin{align}
    \hat{M}(\bm{k})
    &=
    \hat{G}\bigl(\bm{\omega}(\bm{k})\bigr)\,
    \hat{\gamma}_{\sigma}\bigl(|\bm{\omega}(\bm{k})|\bigr)
    =
    \frac{L^2}{\pi |\bm{k}|^2}\,
    \hat{\gamma}_{\sigma}\bigl(|\bm{\omega}(\bm{k})|\bigr),
\end{align}
and the so-called \emph{structure factor} is defined as
\begin{align}
    \hat{\rho}(\bm{k})
    :=
    \sum_{j=1}^n \rho_j\,
    e^{\tfrac{2\pi}{L}\i\,\bm{k}\cdot \bm{x}_j},
    \qquad
    \rho(\bm{x}) := \sum_{j=1}^n \rho_j\,\delta(\bm{x}-\bm{x}_j).
\end{align}
where $\delta$ is the 3D Dirac delta distribution. Under the charge neutrality condition, the $\bm{k}=\bm{0}$ mode vanishes and is therefore omitted from the sum.

Let $\bm{\phi}^{\text{far}} = (\phi^{\text{far}}_j)_{j=1}^n$ and $\bm{\rho} = (\rho_j)_{j=1}^n$ be the length-$n$ vector of the potential and the source strengths at each $\bx_j$, let $\mathbf{U}\in\mathbb{C}^{|\mathcal{I}_{\bm m}|\times n}$ be the nonuniform Fourier matrix with entries $\mathbf{U}_{\ell j}=e^{\frac{2\pi}{L}\i\,\bm{k}_\ell\cdot\bm{x}_j}$, and let $\mathbf{D}=\mathrm{diag}(V^{-1}\hat M(\bm{k}_\ell))$, with the $\bm{k}=\bm{0}$ entry set to zero. Then we can write \eqref{eq:Fsum} in matrix form as
\begin{align}\label{eq:direct_matrix}
    \bm{\phi}^{\text{far}} \;\approx\; \mathbf{U}^{*}\, \mathbf{D}\, \mathbf{U}\, \bm{\rho}.
\end{align}
Direct evaluation of \eqref{eq:direct_matrix} using full Fourier matrices is clearly $\mathcal{O}(|\mathcal{I}_{\bm{m}}|\,n)$. To balance the real-space error, $|\mathcal{I}_{\bm{m}}|$ cannot be chosen too small. For the original Gaussian Ewald split and uniformly random particle locations, the minimal scaling is $|\mathcal{I}_{\bm{m}}|\sim n^{1/2}$, yielding a total cost of $\mathcal{O}(n^{3/2})$ \cite{hockneyComputerSimulationUsing2021}.

The steps above are summarized in Algorithm~\ref{alg:FastEwaldGeneral}.

\begin{alg}[Direct Fourier-space summation]\label{alg:FastEwaldGeneral}
\ \\[4pt]
\textit{Comment:} The radial mollifier $\gamma_{\sigma}$, with width parameter $\sigma$ is chosen based on the cutoff $r_c$, and the radial Green's function $G$ are given,
such that the mollified kernel satisfies
$\hat{M}(\bk)=\hat G(|\bomega(\bk)|)\,\hat\gamma_{\sigma}(|\bomega(\bk)|)$.
Matrix forms are stated in brackets.
\smallskip
\begin{algorithmic}[1]
\REQUIRE{$\{\bm{x}_j\}_{j=1}^n$, $\{\rho_j\}_{j=1}^n$, $L$, $m$\smallskip}
\ENSURE{$\{\phi_j^{\text{far}}\}_{j=1}^n$\medskip}

\STATE{[$\hat{\bm{\rho}} = \mathbf{U}\, \bm{\rho}$] Compute the structure factor:
\begin{align*}
    \hat{\rho}(\bk) = \sum_{j=1}^n \rho_j\, e^{\tfrac{2\pi}{L}\i\, \bm{k} \cdot \bm{x}_j}, 
    \qquad \bk\in \mathcal{I}_{\bm m}.
\end{align*}}

\STATE{[$\hat{\bm{\phi}}^{\text{far}} := \mathbf{D}\, \hat{\bm{\rho}}$] Apply diagonal scaling:
\begin{align*}
    \hat{\phi}^{\text{far}}(\bk) :=
    \begin{cases}
        \dfrac{1}{V}\, \hat{M}(\bm{k})\, \hat{\rho}(\bk), & \bm{k}\in \mathcal{I}_{\bm m}, \\
        0, & \bm{k}=\bm{0}.
    \end{cases}
\end{align*}}

\STATE{[$\bm{\phi}^{\text{far}} \approx \mathbf{U}^{*}\, \hat{\bm{\phi}}^{\text{far}}$] Evaluate the adjoint transform:
\begin{align*}
    \phi_i^{\text{far}} \approx \sum_{\bk\in\mathcal{I}_{\bm{m}}} \hat{\phi}^{\text{far}}(\bk)\, 
    e^{-\tfrac{2\pi}{L}\i\, \bm{k} \cdot \bm{x}_i},\qquad i=1,\ldots,n.
\end{align*}}
\end{algorithmic}
\end{alg}

\subsection{Fast Ewald summation}\label{sec:SE}

In fast Ewald methods, the discrete Fourier transforms of parts~1 and~3 in
Algorithm~\ref{alg:FastEwaldGeneral} are accelerated using FFTs, employing
similar components to those used in type 1 and type 2 nonuniform FFTs (NUFFT)~\cite{duttFastFourierTransforms1993a,greengardAcceleratingNonuniformFast2004}.

Let $\tilde\varphi$ be a periodized window function whose support is, up to possibly a small numerical tolerance, contained in $[-\alpha,\alpha]^3$. Using
$\tilde\varphi$, we spread the $n$ source strengths $\{\rho_j\}_{j=1}^n$ onto
the uniform grid $\{h\bm{l}\}_{\bm{l}\in\mathcal{I}_{\bm{m}}}$ with spacing
$h=L/m$, via
\begin{align}\label{eq:spread_def}
    \ai_{\bm{l}}
    :=
    \sum_{j=1}^n \rho_j\,\tilde\varphi(\bm{x}_j-h\bm{l}),
    \qquad
    \bm{l}\in\mathcal{I}_{\bm{m}}.
\end{align}
In practice, the sum is evaluated only over the $P^3$ grid points in the support
of $\tilde\varphi$ that yield nonzero contributions, giving a spreading cost of
$\mathcal{O}(P^3 n)$.

Applying the FFT to $\{\ai_{\bm{l}}\}_{\bm{l}\in\mathcal{I}_{\bm{m}}}$ yields
\begin{align}
    \hat{\ai}_{\bm{k}}
    =
    \sum_{\bm{l}\in\mathcal{I}_{\bm{m}}}
        \ai_{\bm{l}}\,
        e^{\frac{2\pi}{L}\i\,\bm{k}\cdot(h\bm{l})},
    \qquad
    \bm{k}\in\mathcal{I}_{\bm{m}},
\end{align}
which approximates $\hat{\rho}(\bm{k})\,\hat\varphi(\bm{k})$. To recover an
approximation to $\hat{\rho}(\bm{k})$, it is therefore natural to multiply by
$1/\hat\varphi(\bm{k})$, undoing the convolution introduced by the window.
Since this multiplication depends only on~$\bm{k}$, it is incorporated into the
diagonal scaling step of the algorithm (corresponding to step~2 of
Algorithm~\ref{alg:FastEwaldGeneral}).

Thus, step~1 of Algorithm~\ref{alg:FastEwaldGeneral} is approximated using an FFT in
$\mathcal{O}(P^3 n + |\mathcal{I}_{\bm{m}}|\log |\mathcal{I}_{\bm{m}}|)$.
We emphasize that no oversampling is applied; oversampling would normally be introduced by truncation (or zero-padding) of Fourier modes. In the present setting we assume that the aliasing errors are sufficiently damped by the diagonal scaling; see
Remark~\ref{rem:aliasing_oversampling}.

Because step~3 of the direct algorithm is the adjoint of step~1, the
approximation of step~3 uses the same deconvolution factor,
$1/\hat\varphi(\bm{k})$. Thus the total deconvolution factor is
$1/\hat\varphi(\bm{k})^2$, which is included in the diagonal scaling. The resulting diagonal scaling step in the fast Ewald method is therefore
\begin{align}\label{eq:diag_scale}
    \hat{\aii}_{\bm{k}}
    :=
    \frac{1}{V}\,
    \frac{\hat{M}(\bm{k})}{\hat\varphi(\bm{k})^2}\,
    \hat{\ai}_{\bm{k}},
    \qquad
    \bm{k}\in\mathcal{I}_{\bm{m}}\setminus\{\bm{0}\}.
\end{align}

Now, the inverse FFT yields $\{\aii_{\bm{l}}\}_{\bm{l}\in\mathcal{I}_{\bm{m}}}$, which are then interpolated back to the
particle locations from the uniform grid, yielding the final approximation to the far-field potential:
\begin{align}\label{eq:FastEwaldSteps34}
    \phi_i^{\text{far}}
    \approx
    \sum_{\bm{l}\in\mathcal{I}_{\bm{m}}}
    \aii_{\bm{l}}\,\tilde\varphi(\bm{x}_i-h\bm{l}).
\end{align}

Equations \eqref{eq:spread_def}--\eqref{eq:FastEwaldSteps34} may be written
compactly in matrix form as
\begin{align}\label{eq:matform}
    \bm{\phi}^{\text{far}}
    \approx
    \mathbf{S}^{\!T}\,\mathbf{F}^{-1}\,\mathbf{C}\,\mathbf{D}\,\mathbf{F}\,\mathbf{S}\,\bm{\rho},
\end{align}
where $\mathbf{S}$ denotes spreading, $\mathbf{F}$ and $\mathbf{F}^{-1}$ the
forward and inverse FFTs, $\mathbf{D}=V^{-1}\mathrm{diag}(\hat M(\bm{k}))$,
$\mathbf{C}=\mathrm{diag}(\hat\varphi(\bm{k})^{-2})$, and $\mathbf{S}^{T}$
denotes interpolation.

Thus, the overall cost of evaluating the Fourier-space component is
\begin{align}
\mathcal{O}\!\left(P^3 n + |\mathcal{I}_{\bm{m}}|\log |\mathcal{I}_{\bm{m}}|\right),
\end{align}
that is, linear in the number of particles plus the FFT cost, with $P$ the
window support size. For typical parameter choices with
$|\mathcal{I}_{\bm{m}}|\propto n$, this yields an overall complexity of
$\mathcal{O}(n\log n)$.

The steps of the fast Ewald method, as presented above, are summarized in Algorithm~\ref{alg:SE_Fourier} (that is also shown in a similar form in \cite{liangAcceleratingFastEwald2025}).

\begin{alg}[Fast Fourier-space summation]\label{alg:SE_Fourier}
\ \\[4pt]
\textit{Comment:} The radial mollifier $\gamma_{\sigma}$, with width parameter $\sigma$ chosen based on the cutoff $r_c$, the radial Green's function $G$ such that
$\hat{M}(\bk)=\hat G(|\bomega(\bk)|)\,\hat\gamma_{\sigma}(|\bomega(\bk)|)$, and the window function $\varphi$ whose periodization is used for spreading and interpolation, are given. Matrix forms are stated in brackets.
\smallskip
\begin{algorithmic}[1]
\REQUIRE{$\{\bm{x}_j\}_{j=1}^n$, $\{\rho_j\}_{j=1}^n$, $L$, $m$\smallskip}
\ENSURE{$\{\phi_j^{\text{far}}\}_{j=1}^n$\medskip}

\STATE[$\bm{\ai}:=\mathbf{S}\,\bm{\rho}$] Spread to the uniform grid of spacing $h=L/m$:
\begin{align*}
    \ai_{\bm l}:=\sum_{j=1}^n \rho_j\,\tilde\varphi(\bm{x}_j-h\bm l),
    \qquad \bm l\in\mathcal I_{\bm m}.
\end{align*}

\STATE[$\hat{\bm{\ai}}=\mathbf{F}\,\bm{\ai}$] 3D FFT on the grid:
\begin{align*}
    \hat \ai_{\bm k}
    =
    \sum_{\bm l\in\mathcal I_{\bm m}}
    \ai_{\bm l}\,e^{\frac{2\pi}{L}\i\,\bm k\cdot(h\bm l)},
    \qquad \bm k\in\mathcal{I}_{\bm{m}}.
\end{align*}

\STATE[$\hat{\bm{\aii}}:=\mathbf{C}\,\mathbf{D}\,\hat{\bm{\ai}}$] Diagonal scaling (omit $\bm k=\bm 0$):
\begin{align*}
    \hat \aii_{\bm k}:=
    \begin{cases}
        \dfrac{1}{V}\,\dfrac{\hat M(\bm k)}{\hat\varphi(\bm k)^2}\,\hat \ai_{\bm k},
        & \bm k\in\mathcal{I}_{\bm{m}}\setminus\{\bm{0}\},\\[2pt]
        0, & \bm k=\bm 0.
    \end{cases}
\end{align*}

\STATE[$\bm{\aii}=\mathbf{F}^{-1}\hat{\bm{\aii}}$] Inverse FFT:
\begin{align*}
    \aii_{\bm l}
    =
    \sum_{\bm k\in\mathcal I_{\bm m}}
    \hat \aii_{\bm k}\,e^{-\frac{2\pi}{L}\i\,\bm k\cdot(h\bm l)},
    \qquad \bm l\in \mathcal{I}_{\bm{m}}.
\end{align*}

\STATE[$\bm{\phi}^{\text{far}}\approx\mathbf{S}^{\!T}\bm{\aii}$] Interpolate to particle locations:
\begin{align*}
    \phi_i^{\text{far}}
    \approx
    \sum_{\bm l\in\mathcal I_{\bm m}}
    \aii_{\bm l}\,\tilde\varphi(\bm x_i-h\bm l),
    \qquad i=1,\ldots,n.
\end{align*}
\end{algorithmic}
\end{alg}

\begin{remark}\label{rem:SPME}
Alternative conventions shift window factors differently, giving equivalent
schemes with modified diagonal scaling (e.g., SPME \cite{essmannSmoothParticleMesh1995}).
\end{remark}

\section{Fast Ewald summation for the Laplace kernel using the first PSWF of order zero}
\label{sec:fast_Ewald_PSWF}

In this section, we present how the first PSWF of order zero, $\psi_0^c$, can be employed both as a mollifier in the kernel split and as a window function for the spreading/interpolation steps of Algorithm \ref{alg:SE_Fourier}.

\subsection{General kernel split for numerical mollifiers with support $[-1,1]$}
\label{sec:gen_Ewald_decomp_1}

Let $\gamma_1$ be an even, nonnegative numerical mollifier with $\spt(\gamma_1) = [-1,1]$ and normalization
$\int_{\R} \gamma_1(u)\,\mathrm{d}u = 1$.
Define a one-dimensional \emph{split function} $\Phi\colon \R \to \R$ by
\begin{align}\label{eq:split-function}
    \Phi(x) := 2 \int_0^x \gamma_1(u)\,\mathrm{d}u.
\end{align}
Scaling $\gamma_1$ to a cutoff radius $r_c > 0$, define
\begin{align}\label{eq:rc_mollifier}
    \gamma_{r_c}(x) := \frac{1}{r_c}\,\gamma_1\!\left(\frac{x}{r_c}\right),
\end{align}
which is also a numerical mollifier with support $[-r_c, r_c]$, inheriting the same properties as $\gamma_1$.

Now, by Lemma~\ref{lem:laplacian}, the three-dimensional mollification of the Laplace Green's function is
\begin{align}\label{eq:M}
    M(\bx) = \frac{\Phi_{r_c}(|\bx|)}{|\bx|},
\end{align}
where
\begin{align}
\Phi_{r_c}(x) := \Phi(x/r_c)
= 2 \int_0^{x/r_c} \gamma_1(v)\,\mathrm{d}v
= 2 \int_0^x \gamma_{r_c}(u)\,\mathrm{d}u.
\end{align}
Its Fourier transform then satisfies
\begin{align}
    \hat{M}(\bwv) = \frac{4\pi}{|\bwv|^2}\,\hat{\gamma}_{r_c}(|\bwv|).
\end{align}

Finally, substituting \eqref{eq:M} in the kernel split $G = M + R$, the residual kernel becomes
\begin{align}\label{eq:R}
    R(\bx) = \frac{1 - \Phi_{r_c}(|\bx|)}{|\bx|}.
\end{align}

\begin{remark}
Because $\gamma_1$ is nonnegative and normalized, we have $0 \le \Phi_{r_c}(x) \le 1$ for all $x \ge 0$,
ensuring that the residual kernel $R(\bx)$ is nonnegative.
\end{remark}

\begin{remark}\label{rem:compactsupport}
Since the unscaled mollifier $\gamma_{1}$ satisfies
$\int_{-1}^{1} \gamma_{1}(u),\mathrm{d}u = 1$, we have $\Phi(1)=1$; hence $\Phi_{r_c}(x)=1$ for all $x \ge r_c$. In this case, the residual kernel \eqref{eq:R} is compactly supported on the ball of radius $r_c$, since $R(\bx)=0$ whenever $|\bx| \ge r_c$. A construction of a split function of this type was first presented for the PSWF in the DMK framework \cite{jiangDualspaceMultilevelKernelsplitting}. 

\end{remark}

\subsection{A PSWF mollifier for kernel splitting}
\label{sec:PSWF_mollifier}

Following Section~\ref{sec:gen_Ewald_decomp_1}, we now construct a PSWF mollifier. We start by defining a mollifier $\gamma^{c_s}_1$, supported on $[-1,1]$, as
\begin{align}
\gamma_1^{c_s}(x) :=
\begin{cases}
\dfrac{\psi_0^{c_s}(x)}{\lambda_0 \psi_0^{c_s}(0)}, & x \in [-1,1],\\[2mm]
0, & x \notin [-1,1],
\end{cases}
\end{align}
where $\psi_0^{c_s}$ denotes the first PSWF with Fourier bandlimit $c_s>0$ (the subscript “s” stands for “split”). Here, $\lambda_0$ is the eigenvalue associated with $\psi_0^{c_s}$, and the normalization
$\lambda_0 \psi_0^{c_s}(0)$ follows from Lemma~\ref{lem:PSWFrelation}, ensuring that
$\int_{\mathbb{R}} \gamma_1^{c_s}(x)\,\mathrm{d}x = 1$.

To obtain a mollifier supported on a general radius $r_c>0$, we scale $\gamma^{c_s}_1$ as
\begin{align}\label{eq:PSWF_mollifier}
    \gamma_{r_c}^{c_s}(x)
    := \frac{1}{r_c}\,\gamma_1^{c_s}\!\left(\frac{x}{r_c}\right)
    =
    \begin{cases}
        \dfrac{\psi_0^{c_s}(x/r_c)}{r_c\,\lambda_0\,\psi_0^{c_s}(0)}, & x \in [-r_c, r_c],\\[2mm]
        0, & x \notin [-r_c, r_c],
    \end{cases}
\end{align}
which inherits the mollifier properties of $\gamma^{c_s}_1$.
In contrast to the Gaussian mollifier, whose width parameter $\sigma$ is free,
the PSWF mollifier is compactly supported on $[-r_c,r_c]$.
Consequently, the effective width parameter coincides with the cutoff radius, that is, $\sigma = r_c$.

By Lemma \ref{lem:PSWFrelation}, the Fourier transform of $\gamma_{r_c}^{c_s}$ can be written as
\begin{align}\label{eq:PSWF_mollifier_ft}
    \hat{\gamma}_{r_c}^{c_s}(\omega)
    = \hat{\gamma}_1^{c_s}(r_c\omega)
    = \frac{\psi_0^{c_s}\!\left(r_c \omega / c_s\right)}{\psi_0^{c_s}(0)}.
\end{align}
For $|\omega|>c_s/r_c$, the argument $r_c\omega/c_s$ lies outside $[-1,1]$, so this formula is no longer directly useful when $\psi_0^{c_s}$ is computed numerically from its representation on $[-1,1]$. In that case, the Fourier transform is instead evaluated from its defining integral:
\begin{align}\label{eq:gammahatintegral}
    \hat{\gamma}_{r_c}^{c_s}(\omega)
    = \int_{-r_c}^{r_c} \gamma_{r_c}^{c_s}(x)\,
      e^{-\i\omega x}\,\mathrm{d}x = \frac{1}{\psi_0^{c_s}(0)}\int_{-1}^1\psi_0^{c_s}(u)e^{-i\omega r_c u}\,\d u.
\end{align}

Using \(\gamma_{r_c}^{c_s}\) and \eqref{eq:split-function}, we can thus denote a PSWF-based split function that is closely related to the split function introduced in the DMK
framework~\cite{jiangDualspaceMultilevelKernelsplitting} as
\begin{align}\label{eq:Phi_cs}
    \Phi^{c_s}_{r_c}(x) 
    := 2\int_0^{x} \gamma^{c_s}_{r_c}(t)\,\mathrm{d}t
    = \frac{2}{\lambda_0\,\psi_0^{c_s}(0)} 
      \int_0^{x/r_c} \psi_0^{c_s}(u)\,\mathrm{d}u, 
      \qquad x \in [0, r_c],
\end{align}
where the integral must be evaluated numerically.

Finally, the mollified kernel used in the fast Ewald algorithm
(Algorithm~\ref{alg:SE_Fourier}) can now be written as
\begin{align}\label{eq:MhatPSWF}
    \hat{M}(\bomega)
    = \frac{4\pi}{|\bomega|^2}\,\hat{\gamma}_{r_c}^{c_s}(|\bomega|),
\end{align}
and the residual kernel for the real-space sum is obtained from the split function \eqref{eq:Phi_cs} via~\eqref{eq:R}.

As stated in Remark~\ref{rem:compactsupport}, the residual kernel is compactly supported by construction. Figure~\ref{fig:Phi} shows $\Phi^{c_s}_{0.5}(x)$ and a cross-section of the corresponding residual kernel for $r_c = 0.5$ and various values of $c_s$, illustrating both the compact support and the dependence on $c_s$.

\begin{figure}[t!]
    \centering
    \includegraphics[width=1\textwidth]{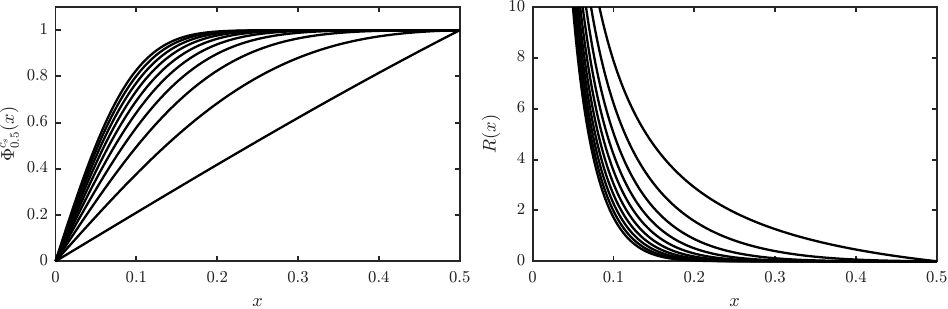}
    \begin{tikzpicture}[overlay]
        \draw[-latex] (-3,2.75) node[below right] {\small $c_s$} -- (-5.75,4.5);
        \draw[-latex] (3.5,2.5) node[right] {\small $c_s$} -- (1.75,1.5);        
    \end{tikzpicture}
    \caption{Split function $\Phi_{r_c}^{c_s}(x)$ (left) and residual kernel $R(x)$ (right) for $r_c = 0.5$ and $c_s \in \{1,6,11,\ldots,49\}$. By construction, the split function satisfies $\Phi_{r_c}^{c_s}(r_c)=1$, and the residual kernel is compactly supported on $[-r_c,r_c]$.}
    \label{fig:Phi}
\end{figure}

Table~\ref{tab:mollifiers} summarizes the PSWF-based mollifier alongside the classical Gaussian mollifier derived in Example~\ref{ex:Gaussian_split}. For a comparison of their spectral decay and support, see also Figure~\ref{fig:PSWFvsGaussian}. 

\begin{table}[t!]
\renewcommand{\arraystretch}{1.5}
\newcommand\Tstrut{\rule{0pt}{5ex}}
\newcommand\TTstrut{\rule{0pt}{6.5ex}}
\centering
\caption{Gaussian and PSWF-based mollifiers. The Gaussian mollifier has width parameter $\sigma$, while the PSWF mollifier is compactly supported with cutoff $r_c$ and bandlimit $c_s$.} 
\label{tab:mollifiers}
\small
\begin{tabular}{|l||l|}
\hline
Definition in real space & Fourier transform \\
\hline
\hline
\Tstrut
$\displaystyle 
  \gamma^{\text{Gauss}}_{\sigma}(x) 
  = \frac{1}{\sqrt{\pi}\,\sigma}\,e^{-x^2/\sigma^2}$ 
& $\displaystyle \hat{\gamma}^{\text{Gauss}}_{\sigma}(\omega) 
   = e^{-\sigma^2\omega^2/4}$ \\[3ex]
\hline
\TTstrut
$\displaystyle \gamma_{r_c}^{\text{PSWF}}(x) = 
  \begin{cases}
    \dfrac{\psi_0^{c_s}(x/r_c)}{r_c\,\lambda_0\,\psi_0^{c_s}(0)}, & x \in [-r_c,r_c]\\[1.5ex]
    0, & \text{otherwise}
  \end{cases}$ 
& $\displaystyle \hat{\gamma}_{r_c}^{\text{PSWF}}(\omega) 
   = 
   \displaystyle\frac{\psi_0^{c_s}\!\left(r_c \omega / c_s\right)}{\psi_0^{c_s}(0)} 
   $
   \\[4ex]
\hline
\end{tabular}
\end{table}

\subsection{A PSWF window function}\label{sec:PSWFwindow}

Analogous to the PSWF mollifier, we introduce two parameters for a PSWF window function: a shape (bandlimit) parameter $c_w>0$ (with $w$ denoting ``window''), which controls spectral concentration and decay, and a width parameter $\alpha>0$, which determines the window half-width.

We define a one-dimensional window function by
\begin{align}
\varphi_1^{c_w}(x) := 
\begin{cases}
\displaystyle\frac{\psi_0^{c_w}(x)}{\psi_0^{c_w}(0)},& x\in [-1,1],\\[5pt] 
0, & \text{otherwise}.
\end{cases}
\end{align}
By Lemma~\ref{lem:PSWFrelation}, its Fourier transform is given by
\begin{align}
    \hat{\varphi}_1^{c_w}(\omega) 
    = \lambda_0\, \frac{\psi_0^{c_w}(\omega/c_w)}{\psi_0^{c_w}(0)}.
\end{align}
A three-dimensional window function $\varphi^{c_w}\colon \R^3 \to \R$ is constructed by rescaling $\varphi_1^{c_w}$ and forming a tensor product of the resulting one-dimensional windows. 
This choice provides a natural extension from one to three dimensions while preserving the optimal concentration properties of the underlying one-dimensional PSWF, and it allows the Fourier transform to factorize into a product of one-dimensional transforms, thereby simplifying both the theoretical analysis and the numerical implementation:
\begin{align}
    \varphi^{c_w}(\bx) 
    := \prod_{i=1}^3 \varphi_1^{c_w}(x_i/\alpha)
    = \begin{cases}
    \displaystyle\prod_{i=1}^3 \frac{\psi_0^{c_w}(x_i/\alpha)}{\psi_0^{c_w}(0)},& \bx \in [-\alpha,\alpha]^3,\\[5pt]
      0,&\text{otherwise}.
    \end{cases}
\end{align}

Since the Fourier transform of a tensor product factorizes into the product of the one-dimensional transforms, we obtain
\begin{align}
\hat{\varphi}^{c_w}(\bomega)
= \prod_{i=1}^3 \alpha\,\hat{\varphi}_1^{c_w}(\alpha \omega_i)
= \left(\frac{\alpha \lambda_0}{\psi_0^{c_w}(0)}\right)^3
\prod_{i=1}^3 \psi_0^{c_w}\!\left(\frac{\alpha \omega_i}{c_w}\right).
\end{align}

The PSWF window function $\varphi^{c_w}$ is by definition compactly supported on $[-\alpha,\alpha]^3$, hence it equals its truncation, $\varphi_t^{c_w}(\bx) = \varphi^{c_w}(\bx)$.

\begin{remark}\label{rem:force}
Due to the tensor-product structure, the gradient of $\varphi^{c_w}$ is
\begin{align}
\frac{\partial\varphi^{c_w}}{\partial x_j}(\bx) 
= \frac{1}{\alpha\, \psi_0^{c_w}(0)^3}\,
\frac{\mathrm{d} \psi_0^{c_w}}{\mathrm{d} s}\!\left(\frac{x_j}{\alpha}\right)\,
\prod_{\substack{i=1\\ i\neq j}}^3 \psi_0^{c_w}\!\left(\frac{x_i}{\alpha}\right),\qquad j=1,2,3,
\end{align}
where the factor $1/\alpha$ arises from the chain rule. The derivative
$\mathrm{d}\psi_0^{c_w}/\mathrm{d}s$ is obtained using the same algorithm as the $\psi_0^{c_w}$ function; see Appendix \ref{app:pswf}.

Thus, for force computations (e.g.\ $\bm{F}_i = -\rho_i \nabla \phi_i$), $\nabla \phi_i$ can be obtained by replacing the last interpolation step (step~5) in Algorithm~\ref{alg:SE_Fourier} with
\begin{align}\label{eq:force}
    \nabla \phi_i^{\text{far}} = \sum_{\bm{l}\in \mathcal{I}_{\bm{m}}} \aii_{\bm{l}}\,\nabla \tilde{\varphi}^{c_w}(\bx_i-h\bm{l}).
\end{align}
\end{remark}

\begin{remark}
To accelerate window evaluation, the 1D domain may be partitioned into $w$ intervals, each approximated by a centered polynomial of degree $p=w+3$, precomputed via a Vandermonde system and then evaluated efficiently using Horner’s rule and SIMD vectorization~\cite{barnettParallelNonuniformFast2019a}. This has been shown to yield a $2$--$3$ times speed-up for spreading and interpolation in 1D for the Kaiser--Bessel function \cite{shamshirgarFastEwaldSummation2021}.

\end{remark}

\subsubsection{Comparison of the PSWF window with Gaussian and B-spline windows}\label{sec:compwindow}

\begin{table}[t!]
\renewcommand{\arraystretch}{1.5}
\newcommand\Tstrut{\rule{0pt}{5.5ex}}
\newcommand\TTstrut{\rule{0pt}{6.5ex}}
\newcommand\Bstrut{\rule[-0.9ex]{0pt}{0pt}}
\centering
\caption{Gaussian, B-spline, and PSWF windows in one dimension. In practice, three-dimensional tensor-product windows $\varphi(\bx)=\prod_{i=1}^3 \varphi(x_i)$ are used. Here $c_g>0$ denotes the Gaussian shape parameter, $\alpha$ the window half-width, and $P$ the B-spline order (which coincides with the support size). For the Gaussian window, the commonly used free-space form is listed.}
\label{tab:reference_windows}
\small
\begin{tabular}{|l||l|}
\hline
Definition in real space & Fourier transform\\
\hline
\hline
\Tstrut
$\displaystyle\varphi^{\text{Gauss}}(x) = e^{-c_g(x/\alpha)^2}$ 
&$\displaystyle\hat{\varphi}^{\text{Gauss}}(\omega) = 
   \alpha \sqrt{\frac{\pi}{c_g}}\; e^{-\omega^2\alpha^2/(4c_g)}$ \\[3ex]
\hline
\Tstrut
$\displaystyle 
\varphi^{\text{B-spline}}(x)=B_P(x/h)$
& $\displaystyle 
\hat{\varphi}^{\text{B-spline}}(\omega)=h\,\hat B_P(h\omega)$ \\[2ex]
\text{where} & \text{where}\\
\qquad $\displaystyle 
B_1(x):=\begin{cases} 
        1, & x\in[-\tfrac12,\tfrac12)\\[1ex]
        0, & \text{otherwise}
      \end{cases}$ 
& \qquad $\displaystyle 
\hat B_P(\omega):=\left(\frac{\sin(\omega/2)}{\omega/2}\right)^{P}$\\[3.5ex]
\qquad $\displaystyle 
B_{P+1}(x):=(B_P * B_1)(x)$  & \\[3ex]
\hline
\TTstrut
$\displaystyle 
\varphi^{\text{PSWF}}(x)
=\begin{cases}
\dfrac{\psi_0^{c_w}(x/\alpha)}{\psi_0^{c_w}(0)}, & x\in [-\alpha,\alpha]\\[1.5ex]
0, & \text{otherwise}
\end{cases}$
& $\displaystyle 
\hat{\varphi}^{\text{PSWF}}(\omega)
=
\displaystyle\alpha\,\lambda_0\,\frac{\psi_0^{c_w}(\alpha \omega / c_w)}{\psi_0^{c_w}(0)}
$\\[4ex]
\hline
\end{tabular}
\end{table}

Table~\ref{tab:reference_windows} summarizes the PSWF, Gaussian, and B-spline windows. The Gaussian window is presented here in its non-truncated (free-space) form. In practice, however, the window is truncated in real space, but rather than defining the Gaussian window explicitly as a truncated function, most implementations continue to use the free-space expressions for both the window and its Fourier transform \cite{lindboSpectralAccuracyFast2011,shamshirgarFastEwaldSummation2021}. This practice introduces truncation errors in the spreading and interpolation steps for the Gaussian window—errors that do not arise for B-spline or PSWF windows, which are compactly supported by construction.

B-splines have support width $P$, equal to the spline order, and have Fourier transform $\hat\varphi(\omega)=h\,\hat B_P(h\omega)$, where $h$ denotes the grid spacing and $\hat B_P(\omega)=(\sin(\omega/2)/(\omega/2))^{P}$. Consequently, their spectral decay is algebraic. For a fixed real-space width $\alpha$, B-splines therefore require either a finer grid or a higher order $P$ to match the accuracy of the PSWF window; increasing $P$ enlarges the support and increases the cost of spreading and interpolation. The widely used SPME method \cite{essmannSmoothParticleMesh1995}---implemented, for example, in GROMACS~\cite{vanderspoelGROMACSFastFlexible2005}---employs non-centered cardinal B-splines for spreading and interpolation and deconvolves their spectrum through a modified diagonal scaling (see Remark~\ref{rem:SPME}). However, the underlying algebraic spectral decay remains unchanged.

\section{Theoretical error analysis for the PSWF split and window functions}
\label{sec:error_analysis}

In this section, we present a rigorous theoretical analysis of the error sources arising in fast Ewald summation using PSWF-based functions.
We derive absolute $L^2$ error estimates associated with both the PSWF mollifier and the PSWF window function introduced in Section~\ref{sec:fast_Ewald_PSWF}. 
For the purpose of parameter selection---where fully rigorous bounds are unnecessarily detailed---we also derive simplified, closed-form error models that are well suited for practical use.

The squared RMS (root-mean-square) error is measured in the volume-normalized $L^2$ norm,
\begin{align}
\frac{1}{V}\|f_{\text{ref}}-f\|^2_{L^2(\Omega)}
\;=\; \frac{1}{V}\int_{\Omega} |f_{\text{ref}}(\bx)-f(\bx)|^2\,\mathrm{d}\bx
\;\approx\; \frac{1}{n}\sum_{i=1}^n |f_{\text{ref}}(\bx_i)-f(\bx_i)|^2,
\end{align}
where the approximation is justified by the law of large numbers when the particle locations
$\{\bx_i\}_{i=1}^n$ are chosen independently and uniformly from $\Omega$.
In this setting, the discrete average converges to the continuous integral as $n \to \infty$.

Let $\phi^{\text{far}}_{t}$ denote the truncated far-field sum $\phi^{\text{far}}$ (see \eqref{eq:Fsum}), and let $\phi^{\text{far}}_{t,h}$ denote an approximation of $\phi^{\text{far}}_{t}$ computed using the fast Ewald algorithm (Algorithm~\ref{alg:SE_Fourier}). In fast Ewald summation, two principal sources of error arise:
(i) truncation of Fourier modes in the Fourier-transformed mollified kernel—referred to as \emph{split errors}—quantified by
$\frac{1}{\sqrt{V}}\|\phi^{\text{far}}-\phi_{t}^{\text{far}}\|_{L^2(\Omega)}$, and
(ii) \emph{approximation errors} arising in the spreading and interpolation steps, quantified by
$\frac{1}{\sqrt{V}}\|\phi_{t}^{\text{far}}-\phi_{t,h}^{\text{far}}\|_{L^2(\Omega)}$.


When the residual kernel is not compactly supported, additional truncation errors arise in the real-space sum. However, for PSWF-based mollifiers, the residual kernel $R$ is compactly supported, and real-space truncation errors are therefore eliminated. This further simplifies the analysis.

By the triangle inequality, the total error can be bounded as
\begin{align}\label{eq:toterror}
\frac{1}{\sqrt{V}}\|\phi^{\text{far}}-\phi_{t,h}^{\text{far}}\|_{L^2(\Omega)}
\leq
\frac{1}{\sqrt{V}}\|\phi^{\text{far}}-\phi_{t}^{\text{far}}\|_{L^2(\Omega)}
+
\frac{1}{\sqrt{V}}\|\phi_{t}^{\text{far}}-\phi_{t,h}^{\text{far}}\|_{L^2(\Omega)}.
\end{align}
Consequently, contributions (i) and (ii) may be analyzed separately.

\subsection{Kernel split truncation errors}\label{sec:dirsumerrors}

Suppose that the errors of the DFT steps in Algorithm \ref{alg:SE_Fourier} (steps 1--2 and 4--5) are negligible. Then the total error depends on the choice of mollifier used in the split and Algorithm \ref{alg:SE_Fourier} turns into Algorithm \ref{alg:FastEwaldGeneral}. In this subsection we analyze these split-errors. 

We start with a small lemma that relates to the structure factor $\hat{\rho}(\bk)$ and the distribution of points. It is important since Ewald summation is most efficient when the points are close to uniformly random. The lemma provides a uniform bound on the Fourier coefficients 
of $\bm{\rho}$ at frequencies above the cutoff $\Kmax$. In particular, it shows 
that the energy contained in any high-frequency mode $\hat{\rho}(\bk)$ cannot exceed a constant multiple of the total $\ell_2$–energy of $\bm{\rho}$.

\begin{lemma}\label{lemma:rho}
For all $\bk \in \Z^3$, the Fourier coefficients of the particle density satisfy
\begin{align}\label{eq:rhoassumption}
    |\hat{\rho}(\bk)|^2 \leq C_{\rho} \, \|\bm{\rho}\|_2^2,
\end{align}
where $0 < C_{\rho} \leq n$.
\end{lemma}

\begin{proof}
By the definition of the discrete Fourier transform, we have
\begin{align}
\hat{\rho}(\bk) = \sum_{j=1}^n \rho_j \, e^{\frac{2\pi}{L} \i\, \bk \cdot \bx_j}.
\end{align}
Applying the Cauchy--Schwarz inequality and using that each exponential factor has modulus one, we obtain
\begin{align}
|\hat{\rho}(\bk)|^2 \leq 
\Big( \sum_{j=1}^n |\rho_j|^2 \Big)
\Big( \sum_{j=1}^n \big| e^{\frac{2\pi}{L}\i \bk \cdot \bx_j} \big|^2 \Big)
= n \, \|\bm{\rho}\|_2^2.
\end{align}
Thus
\begin{align}
C_\rho := \max_{\bk\in\Z^3} \frac{|\hat{\rho}(\bk)|^2}{\|\bm{\rho}\|_2^2} \leq n,
\end{align}
which proves the result.
\end{proof}

\begin{remark}\label{rem:rho}
The constant $C_{\rho}$ reflects the degree of cancellation among the oscillatory exponential terms. 
In the worst case (highly clustered particles), one may have $C_\rho$ as large as $n$. 
For well-distributed configurations (e.g., approximately uniform particle distributions), the exponential vectors corresponding to distinct wavenumbers are nearly orthogonal, leading to substantial cancellations. 
In such cases $C_\rho\approx 1$.
\end{remark}

The following proposition provides an asymptotic bound on the Fourier-tail error induced by the PSWF kernel split.

\begin{proposition}\label{prop:PSWF_split_errors}
Assume that $c_s\ge\sqrt{3}$. Then for $r_c/c_s$ sufficiently small,
\begin{align}\label{eq:splitbnd}
\frac{1}{V}\,\|\phi^{\mathrm{far}}-\phi_t^{\mathrm{far}}\|_{L^2(\Omega)}^2
\le
\frac{8C_\rho\,\|\bm{\rho}\|_2^2\,r_c}{3V c_s}
\left(\frac{\psi_0^{c_s}(1)}{\psi_0^{c_s}(0)}\right)^2
\left(1+\O\!\left(\frac{r_c}{c_s}\right)\right),
\end{align}
where $1\leq C_{\rho}\leq n$ is the concentration constant from Lemma~\ref{lemma:rho}.
\end{proposition}

\begin{proof}
By Parseval’s theorem,
\begin{align}
\frac{1}{V}\|\phi^{\mathrm{far}} - \phi_t^{\mathrm{far}}\|_{L^2(\Omega)}^2
&=
\frac{1}{V^2}
\sum_{\bk\notin \mathcal I_{\bm{m}}}
|\hat{\rho}(\bk)|^2\,|\hat M(\bomega(\bk))|^2\nonumber\\
&\le
\frac{1}{V^2}
\sum_{|\bk|>K_{\max}}
|\hat{\rho}(\bk)|^2\,|\hat M(\bomega(\bk))|^2.
\end{align}
Applying Lemma~\ref{lemma:rho} yields
\begin{align}\label{eq:L2est-main}
\frac{1}{V}\|\phi^{\mathrm{far}} - \phi_t^{\mathrm{far}}\|_{L^2(\Omega)}^2
\le
\frac{C_\rho \|\bm{\rho}\|_2^2}{V^2}
\sum_{|\bk|>\Kmax} |\hat M(\bomega(\bk))|^2.
\end{align}
It therefore remains to estimate the tail sum.

Recall that
\begin{align}
\hat{M}(\bomega(\bk))
=
\frac{4\pi}{|\bomega(\bk)|^{2}}
\frac{\psi_0^{c_s}(r_c |\bomega(\bk)|/c_s)}{\psi_0^{c_s}(0)} .
\end{align}
For $|\bk|>\Kmax$ we have $|\bomega(\bk)|>\omega_{\max}=c_s/r_c$, and hence $r_c|\bomega(\bk)|/c_s>1$.
Therefore, by Lemma \ref{lem:tailbnd}, which gives $|\psi_0^{c_s}(s)|\le \psi_0^{c_s}(1)/s$ for $s\ge 1$ and $c_s\ge\sqrt{3}$, we obtain
\begin{align}
|\hat M(\bomega(\bk))|
\le
\frac{4\pi c_s}{r_c}
\frac{\psi_0^{c_s}(1)}{\psi_0^{c_s}(0)}\frac{1}{|\bomega(\bk)|^3}.
\end{align}
Squaring gives
\begin{align}
|\hat M(\bomega(\bk))|^2
\le
\frac{16\pi^2c_s^2}{r_c^2}
\left(\frac{\psi_0^{c_s}(1)}{\psi_0^{c_s}(0)}\right)^2
\frac{1}{|\bomega(\bk)|^6}.
\end{align}
Summing over the lattice points $\bomega(\bk)=(2\pi/L)\bk$ with $|\bk|>\Kmax$ and estimating the resulting discrete sum by a radial integral via Lemma~\ref{lem:cubecover}, we obtain
\begin{align}
\sum_{|\bk|>\Kmax} |\hat M(\bomega(\bk))|^2
&\le
\frac{16\pi^2c_s^2}{r_c^2}
\left(\frac{\psi_0^{c_s}(1)}{\psi_0^{c_s}(0)}\right)^2
\left(\frac{L}{2\pi}\right)^3
\int_{|\bomega|>\omega_{\max}-\tfrac{\sqrt{3}\pi}{L}}
\frac{1}{|\bomega|^6}\,\d\bomega \nonumber\\
&=
\frac{8V c_s^2}{r_c^2}
\left(\frac{\psi_0^{c_s}(1)}{\psi_0^{c_s}(0)}\right)^2
\int_{\omega_{\max}-\tfrac{\sqrt{3}\pi}{L}}^\infty \frac{1}{\omega^{4}}\,\d\omega \nonumber\\
&\le
\frac{8Vc_s^2}{3r_c^2}
\left(\frac{\psi_0^{c_s}(1)}{\psi_0^{c_s}(0)}\right)^2
\frac{1}{(\omega_{\max}-\tfrac{\sqrt{3}\pi}{L})^3}.\label{eq:hatMsumbound}
\end{align}
In the equality step we passed to spherical coordinates. Since $\omega_{\max}=c_s/r_c$, a Taylor expansion yields
\begin{align}
\frac{1}{(\omega_{\max}-\tfrac{\sqrt{3}\pi}{L})^3}
=
\frac{r_c^3}{c_s^3}
\left(1+\O\!\left(\frac{r_c}{c_s}\right)\right).
\end{align}
Substituting this estimate into \eqref{eq:hatMsumbound} gives
\begin{align}
\sum_{|\bk|>\Kmax} |\hat M(\bomega(\bk))|^2
\le
\frac{8Vc_s^2}{3r_c^2}
\left(\frac{\psi_0^{c_s}(1)}{\psi_0^{c_s}(0)}\right)^2
\frac{r_c^3}{c_s^3}
\left(1+\O\!\left(\frac{r_c}{c_s}\right)\right).
\end{align}
Finally, inserting this bound into \eqref{eq:L2est-main} completes the proof.
\end{proof}

\begin{remark}
A more refined heuristic than the bound \eqref{eq:splitbnd} is obtained by assuming that
\begin{align}\label{eq:psiapprox}
|\psi_{0}^{c_s}(s)|
\approx
\frac{2\psi_0^{c_s}(1)}{\lambda_0}
\frac{|\sin(c_s s)|}{c_s s},
\qquad s>1,
\end{align}
as suggested by Lemma~\ref{lem:exp}. Using $|\sin(c_s s)|\le 1$ and the lower bound
$\lambda_0 > \sqrt{2\pi/c_s}$ yields
\begin{align}
|\psi_{0}^{c_s}(s)|
\lesssim
\sqrt{\frac{2}{\pi c_s}}
\frac{\psi_0^{c_s}(1)}{s},
\qquad s>1.
\end{align}
Replacing $|\psi_0^{c_s}(s)|\le \psi_{0}^{c_s}(1)/s$ by this heuristic estimate in the proof of Proposition~\ref{prop:PSWF_split_errors} leads to the approximation
\begin{align}\label{eq:splitapprox}
\frac{1}{V}\|\phi^{\mathrm{far}}-\phi_t^{\mathrm{far}}\|_{L^2(\Omega)}^2
\approx
\frac{16\,C_{\rho}\|\bm{\rho}\|_2^2\,r_c}{3\pi V\,c_s^2}
\left(\frac{\psi_0^{c_s}(1)}{\psi_0^{c_s}(0)}\right)^2.
\end{align}
Compared with \eqref{eq:splitbnd}, the heuristic estimate \eqref{eq:splitapprox}
gains an additional factor of $c_s^{-1}$.
\end{remark}

\subsubsection{A closed form approximation for the split truncation error}
\label{sec:spliterror}

For practical parameter selection, we use the heuristic model \eqref{eq:splitapprox}.
We approximate the ratio $\psi_0^{c_s}(1)/\psi_0^{c_s}(0)\approx A_2\, c_s^{1/2} e^{-c_s}$, where $A_2\approx 3.42$, obtained in \eqref{eq:asymp3} (see Appendix~\ref{sec:PSWF_appr}). This yields
\begin{align}\label{eq:PSWF_model}
\frac{1}{\sqrt{V}} \|\phi^{\mathrm{far}} - \phi^{\mathrm{far}}_t\|_{L^2(\Omega)}
\approx A_s \frac{\|\bm{\rho}\|_2}{\sqrt{V}} 
\sqrt{r_c}\, c_s^{-1/2} e^{-c_s},
\end{align}
where
\begin{align}
A_s := \frac{4}{\sqrt{3\pi}}\sqrt{C_{\rho}}\,A_2 \approx 4.46\sqrt{C_{\rho}}.
\end{align}
Assuming that the particle positions are reasonably well distributed (i.e., without significant clustering), we expect $C_{\rho}>1$, but close to $1$, as discussed in Remark~\ref{rem:rho}. 

Since $C_{\rho}$ is not known precisely, we treat $A_s$ as a mildly adjustable prefactor. Numerical experiments (see Figure~\ref{fig:n_rc_dependency}) indicate that the choice $A_s\approx 5$ provides a good fit in practice, corresponding to $\sqrt{C_{\rho}}\approx 1.1$.

\subsection{Aliasing errors due to window function approximations}\label{sec:windowfunerrors}

Recall that, in fast Ewald summation, the discrete Fourier transform steps are approximated by convolution with a window function (spreading and interpolation), followed by application of the FFT. These components are closely related to those used in nonuniform FFTs (NUFFTs). Window functions may either be truncated or compactly supported by construction, and the errors introduced by these approximations contribute to the total error together with the split error analyzed in Section \ref{sec:dirsumerrors}.

In the NUFFT literature, one-dimensional error estimates for various window functions have been derived in
\cite{steidlNoteFastFourier1998,pottsFastFourierTransforms2001,barnettAliasingError$expbeta2021,beylkinFastFourierTransform1995}.
A key parameter in these analyses is the oversampling factor, which controls the magnitude of aliasing errors. Fast Ewald methods, however, are typically implemented without oversampling \cite{lindboSpectralAccuracyFast2011,shamshirgarFastEwaldSummation2021}, and we adopt that convention here. Moreover, fast Ewald summation is inherently a three-dimensional problem.

The sources of error arising from the NUFFT components are analogous to those in fast Ewald summation. In practice, all window functions are truncated. As discussed in Section \ref{sec:compwindow} for the Gaussian window, the use of the untruncated Fourier transform introduces additional errors in the spreading and interpolation steps.
 These are commonly referred to as \emph{truncation errors}. A second source of error arises from \emph{aliasing}, which is caused by discrete sampling of the window function. Sampling induces a periodization in Fourier space via the Poisson summation formula, leading to overlapping frequency contributions \cite{plonkaNumericalFourierAnalysis2018}.

The PSWF window function is compactly supported by definition and admits an exact Fourier representation for all $\bk\in\mathcal{I}_{\bm{m}}$. Consequently, window truncation errors are absent, and only aliasing errors need to be analyzed. To avoid ambiguity in the more general results, we retain the subscript $t$ and write $\tilde{\varphi}_t$ (like we define the PSWF window in Section \ref{sec:PSWFwindow}) to emphasize that the window is both periodic and compactly supported.

From the NUFFT literature, it is known that aliasing errors in the spreading and interpolation steps admit similar estimates
\cite{steidlNoteFastFourier1998}. Therefore, we restrict our attention to
aliasing errors arising from the interpolation step.

Recall that the interpolation step in Algorithm~\ref{alg:SE_Fourier}
approximates a truncated function on the form
\begin{align}\label{eq:interpol_1}
    f(\bx) = \sum_{\bk\in \mathcal{I}_{\bm{m}}} \hat{f}_{\bk}\,
    e^{-\frac{2\pi}{L}\,\i\, \bk\cdot \bx}.
\end{align}
It is approximated using translates (convolution) of a window function,
\begin{align}\label{eq:interpol_2}
    f(\bx) \approx f_h(\bx) :=
    \sum_{\bm{l}\in \mathcal{I}_{\bm{m}}} g_{\bm{l}}\,
    \tilde{\varphi}_t(\bx-h\bm{l})
    = \sum_{\bk\in\Z^3}\hat{g}_{\bk}\,c_{\bk}(\tilde{\varphi}_t)\,
      e^{-\frac{2\pi}{L}\,\i\, \bk\cdot\bx}.
\end{align}
Here $\hat g_{\bk}$ denotes the discrete Fourier coefficients of the grid values
$\{g_{\bm l}\}_{\bm l\in\mathcal I_{\bm m}}$ on the $\bm m$-grid, extended
periodically in frequency, so that
\begin{align}
    \hat g_{\bk+m\bm r} = \hat g_{\bk},
    \qquad \bk\in\mathcal I_{\bm m},\ \bm r\in\Z^3.
\end{align}
We choose
\begin{align}\label{eq:interpol_3}
    \hat{g}_{\bk} :=
        \dfrac{\hat{f}_{\bk}}{c_{\bk}(\tilde{\varphi}_t)},\quad
        \bk\in \mathcal{I}_{\bm{m}},
\end{align}
with the convention that $c_{\bk}(\tilde{\varphi}_t)\neq 0$ for
$\bk\in\mathcal{I}_{\bm{m}}$.
This choice can be interpreted as a deconvolution by
$c_{\bk}(\tilde{\varphi}_t)$ and is the one employed in Algorithm \ref{alg:SE_Fourier}. It has been shown to be accurate even in the presence of
aliasing \cite{nestlerAutomatedParameterTuning2016}. This leads to the
following result.

\begin{lemma}[Aliasing error relation]\label{lem:aliasing}
Suppose that $f$ and $f_h$ are given by \eqref{eq:interpol_1} and
\eqref{eq:interpol_2}, with $\hat{g}_{\bk}$ defined by
\eqref{eq:interpol_3}. Assume further that
$c_{\bk}(\tilde{\varphi}_t)\neq 0$ for all $\bk\in\mathcal{I}_{\bm{m}}$.
Then
\begin{align}\label{eq:aliasing_err_1}
    \frac{1}{V}\|f-f_h\|_{L^2(\Omega)}^2
    \;=\; \sum_{\bk\in\mathcal{I}_{\bm{m}}}|\hat{f}_{\bk}|^2
    \sum_{\bm{r}\in\Z^3\setminus\{\bm{0}\}}
    \frac{|c_{\bk+m\bm{r}}(\tilde{\varphi}_t)|^2}
         {|c_{\bk}(\tilde{\varphi}_t)|^2}.
\end{align}
\end{lemma}

\begin{proof}
We have
\begin{align}
    f_h(\bx) &= \sum_{\bk\in\Z^3}\hat{g}_{\bk}\,c_{\bk}(\tilde{\varphi}_t)\,
      e^{-\frac{2\pi}{L}\,\i\, \bk\cdot\bx} \nonumber\\
    &= \sum_{\bk\in\mathcal{I}_{\bm{m}}}\hat{g}_{\bk}\,c_{\bk}(\tilde{\varphi}_t)\,
      e^{-\frac{2\pi}{L}\,\i\, \bk\cdot\bx}
     + \sum_{\bk\in\mathcal{I}_{\bm{m}}}\sum_{\bm{r}\in\Z^3\setminus\{\bm{0}\}}
       \hat{g}_{\bk+m\bm{r}}\,c_{\bk+m\bm{r}}(\tilde{\varphi}_t)\,
       e^{-\frac{2\pi}{L}\,\i\, (\bk+m\bm{r})\cdot\bx}.\label{eq:alerrrel}
\end{align}
Subtracting \eqref{eq:alerrrel} from \eqref{eq:interpol_1} and using the
periodicity $\hat g_{\bk+m\bm r}=\hat g_{\bk}$ yields
\begin{align}
    f(\bx)-f_h(\bx)
    &= \sum_{\bk\in\mathcal I_{\bm m}}
       \bigl(\hat f_{\bk}-\hat g_{\bk}c_{\bk}(\tilde{\varphi}_t)\bigr)
       e^{-\frac{2\pi}{L}\,\i\,\bk\cdot\bx} \nonumber\\
    &\quad
     - \sum_{\bk\in\mathcal I_{\bm m}}\sum_{\bm r\in\mathbb Z^3\setminus\{\bm0\}}
       \hat g_{\bk}\,c_{\bk+m\bm r}(\tilde{\varphi}_t)
       e^{-\frac{2\pi}{L}\,\i\,(\bk+m\bm r)\cdot\bx}.
\end{align}
The two sums are supported on disjoint frequency sets and are therefore
orthogonal in $L^2(\Omega)$. Parseval’s theorem gives
\begin{align}
    \|f-f_h\|_{L^2(\Omega)}^2
    = V\sum_{\bk\in\mathcal I_{\bm m}}
       \bigl|\hat f_{\bk}-\hat g_{\bk}c_{\bk}(\tilde{\varphi}_t)\bigr|^2
     + V\sum_{\bk\in\mathcal I_{\bm m}}\sum_{\bm r\in\mathbb Z^3\setminus\{\bm0\}}
       |\hat g_{\bk}|^2\,|c_{\bk+m\bm r}(\tilde{\varphi}_t)|^2.
\end{align}
Using \eqref{eq:interpol_3}, the first sum vanishes and
$|\hat g_{\bk}|^2 = |\hat f_{\bk}|^2 / |c_{\bk}(\tilde{\varphi}_t)|^2$.
Substitution yields \eqref{eq:aliasing_err_1}.
\end{proof}
\begin{remark}[Oversampling]\label{rem:aliasing_oversampling}
The coefficients $c_{\bk}(\tilde{\varphi}_t)$ typically decay in magnitude as $|\bk|$ increases, possibly with oscillations (as for PSWF). However, this decay alone does not guarantee suppression of aliasing errors. Indeed, for the specific choice $\bk=(m/2,0,0)$ and $\bm r=(-1,0,0)$, we have $\bk+m\bm r =-\bk$.
For this choice, since the coefficients are even, we have
\begin{align}
\frac{c^2_{\bk+m\bm r}(\tilde{\varphi}_t)}
     {c^2_{\bk}(\tilde{\varphi}_t)} = 1.
\end{align}
Thus, in the absence of oversampling, aliasing errors are not uniformly suppressed by the window function alone. Nevertheless, the diagonal scaling in fast Ewald methods---via the $|\hat f_{\bk}|^2$ factor---provides additional damping.
\end{remark}

Using Lemma \ref{lem:aliasing} we can now estimate the aliasing errors of Algorithm \ref{alg:SE_Fourier} for the PSWF mollifier and the PSWF window function.

\begin{proposition}[Aliasing error in interpolation]\label{prop:interpol_est}
Let $\phi_t^{\mathrm{far}}$ be the truncated Fourier-space far-field potential
produced by Algorithm~\ref{alg:FastEwaldGeneral}, using the PSWF mollifier
$\gamma_{r_c}^{c_s}$, with the Fourier series representation
\begin{align}
    \phi_t^{\mathrm{far}}(\boldsymbol{x})
    = \frac{1}{V}\sum_{\boldsymbol{k}\in\mathcal{I}_{\boldsymbol{m}}\setminus\{\boldsymbol{0}\}}
      \hat M(\boldsymbol{k})\,\hat\rho(\boldsymbol{k})\,
      e^{-\frac{2\pi i}{L}\boldsymbol{k}\cdot\boldsymbol{x}} .
\end{align}
Let $\phi_{t,h}^{\mathrm{far}}$ denote its approximation obtained by the
interpolation step of Algorithm~\ref{alg:SE_Fourier}, using the PSWF window
$\tilde\varphi_t^{c_w}$.

For the purpose of the analysis, we employ the deconvolution choice
\begin{align}
    \hat g_{\boldsymbol{k}}
    := \frac{1}{V}\,
       \frac{\hat M(\boldsymbol{k})\,\hat\rho(\boldsymbol{k})}
            {c_{\boldsymbol{k}}(\tilde\varphi_t^{c_w})},
    \qquad \boldsymbol{k}\in\mathcal{I}_{\boldsymbol{m}},
\end{align}
where $c_{\boldsymbol{k}}(\tilde\varphi_t^{c_w})$ denotes the Fourier series
coefficient of the periodized window.

Then the interpolation error satisfies
\begin{align}\label{eq:interpol_est}
    \frac{1}{V}\,
    \|\phi_t^{\mathrm{far}}-\phi_{t,h}^{\mathrm{far}}\|_{L^2(\Omega)}^2
    \;\le\;
    \frac{C_\rho\,L\,\|\boldsymbol{\rho}\|_2^2}
         {V\pi^2\,\psi_0^{c_s}(0)^2}
    \sum_{\boldsymbol{k}\in\mathcal{I}_{\boldsymbol{m}}\setminus\{\boldsymbol{0}\}}
    \frac{\psi_0^{c_s}\!\left(\tfrac{2}{m}|\boldsymbol{k}|\right)^2
    \displaystyle
    \sum_{\boldsymbol{r}\in\mathbb{Z}^3\setminus\{\boldsymbol{0}\}}
    \prod_{i=1}^3
    \psi_0^{c_w}\!\Big(\tfrac{2}{m}(k_i+m r_i)\Big)^2}
    {|\boldsymbol{k}|^4
    \displaystyle\prod_{i=1}^3
    \psi_0^{c_w}\!\left(\tfrac{2}{m}k_i\right)^2},
\end{align}
where $0<C_\rho\le n$ is the constant from Lemma~\ref{lemma:rho}.
\end{proposition}

\begin{proof}
By Lemma~\ref{lem:aliasing}, applied to $\phi_{t,h}^{\text{far}}$, and using
Lemma~\ref{lemma:rho} together with
$\hat{M}(\bk)=\hat{G}(\bk)\,\hat{\gamma}_{r_c}^{c_s}(|\bk|)$, we obtain
\begin{align}
    \frac{1}{V}\|\phi_t^{\text{far}}-\phi^{\text{far}}_{t,h}\|_{L^2(\Omega)}^2
   &= \frac{1}{V^2}\sum_{\bk\in\mathcal{I}_{\bm{m}}\setminus\{\bm{0}\}}
      \big|\hat{M}(\bk)\,\hat{\rho}(\bk)\big|^2
      \sum_{\bm{r}\in\Z^3\setminus\{\bm{0}\}}
      \frac{c_{\bk+m\bm{r}}^2(\tilde{\varphi}_t^{c_w})}
           {c_{\bk}^2(\tilde{\varphi}_t^{c_w})} \nonumber\\
   &\leq \frac{C_\rho}{V^2}\,\|\bm{\rho}\|_2^2
      \sum_{\bk\in\mathcal{I}_{\bm{m}}\setminus\{\bm{0}\}}
      \big|\hat{G}(\bk)\,\hat{\gamma}_{r_c}^{c_s}(|\bk|)\big|^2
      \sum_{\bm{r}\in\Z^3\setminus\{\bm{0}\}}
      \frac{c_{\bk+m\bm{r}}^2(\tilde{\varphi}_t^{c_w})}
           {c_{\bk}^2(\tilde{\varphi}_t^{c_w})}.
\end{align}
By \eqref{eq:fcoeffs},
$c_{\bk}(\tilde{\varphi}_t^{c_w}) = \tfrac{1}{V}\,\hat{\varphi}_t^{c_w}(\bk)$.
Using $\omega_i/\omega_{\max}=2k_i/m$, this ratio can be written as
\begin{align}
\frac{c_{\bk+m\bm{r}}^2(\tilde{\varphi}_t^{c_w})}
     {c_{\bk}^2(\tilde{\varphi}_t^{c_w})}
= \frac{\displaystyle\prod_{i=1}^3
       \psi_0^{c_w}\!\Big(\tfrac{2}{m}(k_i+m r_i)\Big)^2}
       {\displaystyle\prod_{i=1}^3
       \psi_0^{c_w}\!\Big(\tfrac{2}{m}k_i\Big)^2},
\end{align}
where the normalization constants in $\hat{\varphi}_t^{c_w}$ cancel.
With $\omega = 2\pi|\bk|/L$ and $r_c = c_s L/(\pi m)$ we have
\begin{align}
\frac{r_c\,\omega}{c_s} = \frac{2}{m}|\bk|,
\end{align}
so the PSWF argument simplifies as claimed. Substituting
$\hat{G}(\bk)=L^2/(\pi|\bk|^2)$ and the explicit representation of
$\hat{\gamma}_{r_c}^{c_s}$ yields
\begin{align}
    \frac{1}{V}\|\phi_t^{\text{far}}-\phi_{t,h}^{\text{far}}\|_{L^2(\Omega)}^2
    \leq
    \frac{C_\rho\,L^4\|\bm{\rho}\|_2^2}{V^2\pi^2\psi_0^{c_s}(0)^2}
    \sum_{\bk\in\mathcal{I}_{\bm{m}}\setminus\{\bm{0}\}}
    \left|\frac{\psi_0^{c_s}\!\big(\tfrac{2}{m}|\bk|\big)}{|\bk|^2}\right|^2
    \sum_{\bm{r}\in\Z^3\setminus\{\bm{0}\}}
    \frac{\displaystyle\prod_{i=1}^3
    \psi_0^{c_w}\!\Big(\tfrac{2}{m}(k_i+m r_i)\Big)^2}
         {\displaystyle\prod_{i=1}^3
    \psi_0^{c_w}\!\Big(\tfrac{2}{m}k_i\Big)^2}.
\end{align}
Substituting $V=L^3$ then gives the desired estimate~\eqref{eq:interpol_est}.
\end{proof}

\subsubsection{A closed-form approximation for the aliasing error}
\label{sec:aliasingerr_closedform}

We now derive an approximate model for the aliasing error by selecting the dominant terms in the aliasing sum. This approach is inspired by \cite{liangAcceleratingFastEwald2025}, where a simplified aliasing sum for the relative force error, similar in form to \eqref{eq:aliasing-leading} below, is stated.

In~\eqref{eq:interpol_est}, the aliasing sum is dominated by the first nonzero
terms with $|\bm r|=1$, since the PSWF satisfies
$|\psi_0^{c_w}(s)| \leq \psi_0^{c_w}(1)/s$ for $s \ge 1$ (Lemma~\ref{lem:tailbnd}),
implying fast (algebraic) decay of the shifted coefficients as $|\bm r|$ increases. These six contributions—corresponding to one shift in each
coordinate direction in three dimensions—provide the leading-order
term to the aliasing error. We therefore approximate
\eqref{eq:interpol_est} by retaining only these terms:
\begin{align}
\frac{1}{V}\|\phi_t^{\mathrm{far}}-\phi_{t,h}^{\mathrm{far}}\|^2_{L^2(\Omega)} 
&\approx \frac{C_{\rho}L \|\bm{\rho}\|_2^2}{V\pi^2\psi_0^{c_s}(0)^2}
\sum_{\bm{k}\in\mathcal{I}_{\bm{m}}\setminus\{\bm{0}\}}
\frac{\psi_0^{c_s}\!\left(\tfrac{2}{m}|\bm{k}|\right)^2
\displaystyle \sum_{|\bm{r}|=1}
\prod_{i=1}^3\psi_0^{c_w}\!\Big(\tfrac{2}{m}(k_i+m r_i)\Big)^2}
{|\bm{k}|^4\displaystyle\prod_{i=1}^3\psi_0^{c_w}\!\left(\tfrac{2}{m}k_i\right)^2}.
\label{eq:aliasing-leading}
\end{align}

Since $|\bm{k}|\ge |k_i|$, and since $\psi_0^{c}$ is even and nonincreasing on $[0,1]$ and bounded by $\psi_0^{c}(1)$ on $s>1$ by Lemma~\ref{lem:tailbnd}, we obtain for $\bm{k}=(k_1,k_2,k_3)\in\mathcal{I}_{\bm{m}}\setminus\{\bm{0}\}$,
\begin{align}
\psi_0^{c_s}\!\left(\tfrac{2}{m}|\bm{k}|\right)\le 
\psi_0^{c_s}\!\left(\tfrac{2}{m}|k_i|\right),\qquad i\in\{1,2,3\}.
\label{eq:ratio-le-1}
\end{align}
Equality holds for axis-aligned modes of the form
$\bm{k}=s\,k\,\mathbf e_j$, where $s\in\{\pm1\}$, $j\in\{1,2,3\}$, and
$k\in\{1,\dots,K_{\max}\}$ with $K_{\max}=\lfloor m/2\rfloor$.
For such modes,
\begin{align}\label{eq:windowapprox}
\frac{\psi_0^{c_s}\!\left(\tfrac{2}{m}|\bm{k}|\right)^2}
     {\prod_{i=1}^3 \psi_0^{c_w}\!\left(\tfrac{2}{m} k_i\right)^2}
&=
\frac{\psi_0^{c_s}\!\left(\tfrac{2}{m}k\right)^2}
     {\psi_0^{c_w}\!\left(\tfrac{2}{m}k\right)^2\psi_0^{c_w}(0)^4}
\;\approx\; \frac{1}{\psi_0^{c_w}(0)^4},
\end{align}
where the approximation uses that $c_s\approx c_w$, so that
$\psi_0^{c_s}(2k/m)/\psi_0^{c_w}(2k/m)\approx 1$
for the relevant range of $k$.

For $k_i\in \mathcal{I}_m$ and $r_i\in\{\pm1\}$, we have
\begin{align}
\left|\tfrac{2}{m}(k_i + m r_i)\right| \in [1,3]
\qquad\text{and}\qquad 
\left|\tfrac{2}{m}k_i\right| \in (0,1].
\label{eq:shifted-range}
\end{align}
Hence, for $|\bm{r}|=1$, we obtain
\begin{align}\label{eq:windowbnd}
\prod_{i=1}^3\psi_0^{c_w}\!\Big(\tfrac{2}{m}(k_i+m r_i)\Big)^2 
\;\le\; \psi_0^{c_w}(1)^2\,\psi_0^{c_w}(0)^4.
\end{align}
Equality occurs (for even $m$) in the extremal axis-aligned case where
$\bigl|\tfrac{2}{m}(k_j+m r_j)\bigr|=1$, i.e., $|k_j|=K_{\max}$ and
$r_j=-\operatorname{sgn}(k_j)$.

Using \eqref{eq:windowapprox} and \eqref{eq:windowbnd}, and retaining only the axis-aligned modes, we obtain the approximation
\begin{align}
\sum_{\bm{k}\in\mathcal{I}_{\bm{m}}\setminus\{\bm{0}\}}
\frac{\psi_0^{c_s}\!\left(\tfrac{2}{m}|\bm{k}|\right)^2\displaystyle \sum_{|\bm{r}|=1}
\prod_{i=1}^3\psi_0^{c_w}\!\Big(\tfrac{2}{m}(k_i+m r_i)\Big)^2}
{|\bm{k}|^4\displaystyle\prod_{i=1}^3\psi_0^{c_w}\!\left(\tfrac{2}{m}k_i\right)^2}
&\approx
6\,\psi_0^{c_w}(1)^2\sum_{k=1}^{K_{\max}}\frac{1}{k^4}
\;\lesssim\; \psi_0^{c_w}(1)^2\,\frac{\pi^4}{15},
\label{eq:axis-bound}
\end{align}
where the factor $6$ accounts for three coordinate axes and two signs, and
$\sum_{k\ge1}k^{-4}=\pi^4/90$.

Substituting~\eqref{eq:axis-bound} into~\eqref{eq:aliasing-leading} yields the
dimensionless prefactor $C_{\rho}\pi^2/15$, which is of
order unity and independent of $m$, $L$, and $c_w$. We thus arrive at the
approximation
\begin{align}
\frac{1}{\sqrt{V}}
\|\phi_t^{\mathrm{far}}-\phi_{t,h}^{\mathrm{far}}\|_{L^2(\Omega)}
\approx
\frac{\|\bm{\rho}\|_2}{\sqrt{V}}
\frac{\pi}{\sqrt{15}}\,\sqrt{C_{\rho}L}\,
\frac{\psi_0^{c_w}(1)}{\psi_0^{c_w}(0)}.
\end{align}

Approximating $\psi_0^{c_w}(1)/\psi_0^{c_w}(0)\approx A_2\,c_w^{1/2}e^{-c_w}$, with $A_2\approx 3.42$ from \eqref{eq:asymp3} (Appendix~\ref{sec:PSWF_appr}), we obtain
\begin{align}\label{eq:window_err_model}
\frac{1}{\sqrt{V}}
\|\phi_t^{\mathrm{far}}-\phi_{t,h}^{\mathrm{far}}\|_{L^2(\Omega)}
\approx
A_w\,\frac{\|\bm{\rho}\|_2}{\sqrt{V}}\,
\sqrt{L}\,c_w^{1/2}e^{-c_w},
\end{align}
where
\begin{align}
A_w := \frac{\pi}{\sqrt{15}}\,\sqrt{C_{\rho}}\,A_2 \approx 2.8\sqrt{C_{\rho}}.
\end{align}
As for the split-model \eqref{eq:PSWF_model}, we use the same value of $\sqrt{C_\rho}$, treated as a mildly adjustable factor; numerical experiments indicate $\sqrt{C_\rho}\approx 1.1$ for finite, well-distributed particle configurations, yielding $A_w\approx 3.1$.




\section{Parameter selection}
\label{sec:parameter_selection}

This section presents practical guidelines for selecting the PSWF mollifier and window parameters to achieve a prescribed tolerance~$\varepsilon$. The results rely on the error models of Section~\ref{sec:error_analysis}. Unless otherwise noted, we assume neutral, well-distributed particle configurations, 
no oversampling, and parameters $c_s,c_w$ within the fitted ranges of Table~\ref{tab:curvefits}.

For simplicity, we assume that $m$ is even throughout this section; in the odd case, one can replace $m/2$ by $\lfloor m/2\rfloor$.

\subsection{Notes about nondimensionalization and the cuboid setting}
\label{sec:nondimensionalization}

The Ewald parameters depend only on ratios of characteristic lengths, and are therefore invariant under uniform rescaling of the problem. This allows them to be selected using a normalized cutoff, for example \( r_c/L \) in the cubic case.

For the periodic cubic domain \([0,L)^3\), we can define nondimensional coordinates $\bx^* := \bx/L \in [0,1)^3$. In these variables, the computational domain becomes the unit cube.
Periodic functions may then be expanded in the Fourier basis $e^{2\pi \mathrm{i}\,\bk\cdot \bx^*}$, for $\bk \in \mathbb{Z}^3$, which corresponds in physical coordinates to wavevectors \(2\pi(\bk / L)\). In this setting, it is natural to introduce the dimensionless cutoff
\begin{align}\label{eq:rcnondim}
r_c^* := \frac{r_c}{L}, \qquad 0 < r_c^* < 1,
\end{align}
and to express all algorithmic parameters---such as \(c_s\), \(c_w\), and the grid size \(P\)---in terms of \(r_c^*\) and the target tolerance~\(\varepsilon\).
Once these parameters are fixed, physical quantities are recovered by rescaling lengths by~\(L\).

The free-space Laplace kernel satisfies $G(\bx) = 1/|\bx|$, and under the isotropic scaling \(\bx = L\bx^*\) one has
$G(\bx) = \frac{1}{L}\,G(\bx^*)$.
Accordingly, the periodic potential may be written as
\begin{align}\label{eq:phinondim}
\phi(\bx_i) = \frac{1}{L}\,\phi^*(\bx_i^*),\qquad
\phi^*(\bx_i^*)
= \sum_{\br\in\mathbb{Z}^3} \sum_{j=1}^n{\vphantom{\sum}}'
G(|\bx_i^* - \bx_j^* + \br|)\,\rho_j.
\end{align}
Hence, we can always solve on the unit cube, using a dimensionless $r_c$ that accounts for $L\neq 1$, and in the end scale the potential with~$L$.

\begin{remark}[Cuboid domains]
As noted in the introduction, the method extends naturally to cuboid domains. 
Let \(\Omega = \prod_{j=1}^3 [0,L_j)\) and choose the FFT grid so that each direction 
uses \(m_i\) points with a common grid spacing \(h\) satisfying \(h = L_i/m_i\), \(i=1,2,3\). With \(m_i\) even, this corresponds to representing \(m_i/2\) Fourier modes per direction, and the largest resolved wave number is \(\omega_{\max} = \pi m_i/L_i = \pi/h\). 
After introducing the componentwise scaled coordinates \(\bx^* = \bx/L_{\min}\), where \(L_{\min} := \min \bm{L}\), the nondimensionalization proceeds as in the cubic case. 
In particular, the radial cutoff becomes \(r_c^* = r_c/L_{\min}\), in analogy with \eqref{eq:rcnondim}. The algorithm itself is unchanged, and the Ewald parameters depend only on the relative cutoff size, not on the absolute domain dimensions or its aspect ratio.
\end{remark}

\subsection{PSWF split parameters}

Two parameters control the kernel split: the shape parameter~$c_s$ and the cutoff~$r_c$.
The cutoff balances the real- and Fourier-space costs, while $c_s$ determines the number of Fourier modes.
They are related by (see \eqref{eq:PSWF_mollifier_ft})
\begin{align}\label{eq:csrc}
    \frac{c_s}{r_c} = \omega_{\max} = \frac{2\pi}{L} K_{\max}
    = \frac{\pi m}{L}.
\end{align}
Thus, given $c_s$ and $r_c$, it is reasonable to set $m$ to
\begin{align}\label{eq:m_cs_dimless}
    m \approx \left\lceil \frac{Lc_s}{\pi r_c} \right\rceil .
\end{align}

\subsubsection{Choosing $c_s$}

Using the fitted model~\eqref{eq:PSWF_model}, we choose $c_s$ as
\begin{align}\label{eq:cs_dimless}
    c_s \approx \frac{1}{2}W\Big(2\,\big(B_s/\varepsilon\big)^2\Big),\qquad 
    B_s := \frac{1}{\sqrt{V}}A_s\|\bm{\rho}\|_2 \sqrt{r_c},
\end{align}
where $A_s \approx 5$. 
This choice provides accurate parameter selection across the practically relevant range of $c_s$ and yields close agreement with the observed absolute RMS errors.

A simpler, but less accurate alternative, is based on the asymptotic estimate~\cite{jiangDualspaceMultilevelKernelsplitting,liangAcceleratingFastEwald2025}
\begin{align}
    c_s \approx \log(1/\varepsilon),
\end{align}
which captures the correct exponential scaling and has the advantage of being parameter-free. 
It provides a reasonable estimate of $c_s$, although it may slightly overestimate the required bandwidth for small to moderate values of $c_s$ (see Figure~\ref{fig:abserrors}).

In practice, both choices are effective.

\subsubsection{Choosing $r_c$}

In molecular simulations (e.g., GROMACS~\cite{vanderspoelGROMACSFastFlexible2005}), the real-space cutoff~$r_c$ is often set empirically from physical interaction ranges (e.g., \(9\,\text{\AA}\)) rather than chosen optimally for computational efficiency.
In the present formulation, we suggest that $r_c$ should be treated as a tunable parameter balancing the real- and Fourier-space costs; note that $r_c$ is linked to $m$ through \eqref{eq:m_cs_dimless}.
The optimal value of $r_c$ depends on several factors, such as hardware, software, particle number~$n$, and particle distribution, and is best found empirically by measuring the total runtime for a range of~$r_c$ values.

\subsection{Window parameters}

If both the PSWF mollifier and the PSWF window are used, their scaled bandlimits should match, yielding
\begin{align}\label{eq:bandwidth}
    \frac{c_s}{r_c} = \frac{c_w}{\alpha}.
\end{align}
Hence, given $c_s$, $r_c$, and $c_w$, the window half-width should satisfy
\begin{align}\label{eq:alpha}
\alpha = \frac{r_c c_w}{c_s}.
\end{align}
Moreover, the window half-width is $\alpha = PL/(2m)$, which gives
\begin{align}\label{eq:P_dimless}
P = \biggl\lceil \frac{2 m \alpha}{L}\biggr\rceil.
\end{align}
This relation makes the parameters match closely, also after applying rounding to $m$; see Remark~\ref{rem:parameter} below. However, using \eqref{eq:csrc} (assuming $m$ is even) together with $\alpha = Ph/2$, we also obtain the simple relation between $c_w$ and $P$,
\begin{align}\label{eq:cwP}
    c_w = \frac{\pi}{2} P,
\end{align}
showing clearly that $P$ is proportional to $c_w$.

\subsubsection{Choosing $c_w$}

To choose $c_w$, we assume that the split and window errors are balanced. 
Equating \eqref{eq:window_err_model} and \eqref{eq:PSWF_model} yields
\begin{align}\label{eq:cwcs_model_dimless}
    \sqrt{L}\,A_w\,c_w^{1/2}\, e^{-c_w} = \sqrt{r_c}\,A_{s}\,c_s^{-1/2}e^{-c_s}.
\end{align}
Notably, the constants $A_s$ and $A_w$ share the same factors, so that their ratio is independent of $C_\rho$. 
In particular, we have
\begin{align}
    \frac{A_s}{A_w} = \frac{4\sqrt{5}}{\pi^{3/2}}\approx 1.61.
\end{align}

Solving \eqref{eq:cwcs_model_dimless} for $c_w$ using the Lambert $W$ function gives
\begin{align}\label{eq:cwW}
    c_w \approx -\frac{1}{2} W\!\left(-2 B_w^2 \frac{e^{-2c_s}}{c_s}\right),\qquad 
    B_w := \frac{A_s}{A_w}\sqrt{\frac{r_c}{L}}.
\end{align}

\subsection{Parameter selection procedure}

A typical parameter selection procedure for the PSWF/PSWF method is presented in Algorithm~\ref{alg:parameter} below.

\begin{alg}[Parameter selection procedure for the PSWF/PSWF method]\label{alg:parameter}
\ \\[4pt]
\textit{Comment:} The number of particles $n$, the source strengths $\bm{\rho}$, and the box size $L$ are given.
\smallskip
\begin{algorithmic}[1]
\REQUIRE{Tolerance $\varepsilon$ and cutoff radius $r_c$.}
\ENSURE{$c_s$, $c_w$, $m$, $P$, $\alpha$.}
\smallskip
\STATE Solve $c_s = \frac{1}{2}W \big(2\,(B_s/\varepsilon)^2\big)$ for $c_s$, where $B_s := \frac{1}{\sqrt{V}}A_s\|\bm{\rho}\|_2 \sqrt{r_c}$; alternatively, use $c_s = \log(1/\varepsilon)$.
\STATE Solve $c_w = -\frac{1}{2}W\big(-2 B_w^2 \frac{e^{-2c_s}}{c_s}\big)$ for $c_w$, where $B_w := \frac{A_s}{A_w}\sqrt{\frac{r_c}{L}}$.
\STATE Set $\alpha = r_c c_w/c_s$.
\STATE Set $m = \lceil L c_s /(\pi r_c)\rceil$.
\STATE Set $P = \lceil 2\alpha m/L\rceil$.
\end{algorithmic}
\end{alg}


\begin{remark}\label{rem:parameter}
Algorithm~\ref{alg:parameter} yields an absolute RMS error that closely matches the prescribed tolerance~$\varepsilon$
(see Figure~\ref{fig:errorvstol} in Section~\ref{sec:num_results}). 
This agreement arises because the rounding of $P$ and $m$ via
\eqref{eq:P_dimless} and \eqref{eq:m_cs_dimless} is of comparable relative magnitude and therefore effectively balanced.
An alternative approach to selecting $P$ is to use~\eqref{eq:cwP} and set
$P = \lceil (2/\pi)c_w\rceil$.
However, depending on the extent of the rounding applied to $P$, this choice may result in an effective value of $c_w$ that is too small, and consequently in a computed error that deviates more from $\varepsilon$.
\end{remark}



\section{Numerical results}\label{sec:num_results}

The algorithms were implemented in \textsc{MATLAB}, with PSWFs computed to high accuracy 
using \texttt{Chebfun} and its built-in \texttt{pswf} routine~\cite{chebfun}. For 
general-purpose use, the efficient algorithm described in 
Appendix~\ref{app:pswf} and in \cite{osipovEvaluationProlateSpheroidal2014} can be implemented in any high-performance programming language to evaluate PSWF 
functions. Timing results for a parallel implementation in GROMACS are reported 
in~\cite{liangAcceleratingFastEwald2025}.

\subsection{Experimental setup}

Unless otherwise stated, we consider systems of $n = 100$ particles uniformly 
distributed in $[0,1)^3$, generated using MATLAB's \texttt{rand} function. The particle 
charges $\rho_j$ are sampled independently from a standard normal distribution (mean 
$0$, variance $1$) and subsequently shifted to enforce charge neutrality, 
$\sum_{j} \rho_j = 0$. The choice $n = 100$ is sufficient to capture the 
relevant convergence trends; larger systems exhibit the same slopes and 
offsets; see e.g.~Figure \ref{fig:n_rc_dependency}. Unless otherwise stated, we set the cutoff radius to 
$r_c = 0.1$ and the domain size to $L = 1$, which provides reasonable Fourier-space 
resolution for testing purposes. For the split model, we use $A_s \approx 5$, and for the aliasing error model, $A_w \approx 3.1$, as discussed in Section~\ref{sec:error_analysis}.

We compute the approximate far-field potential values 
$\bm{\phi}^{\mathrm{far}} = (\phi^{\mathrm{far}}_j)_{j=1}^n$ using 
Algorithm~\ref{alg:SE_Fourier} and compare them with a reference solution 
$\bm{\phi}^{\mathrm{far}}_{\mathrm{ref}} =
(\phi^{\mathrm{far}}_{\mathrm{ref},j})_{j=1}^n$, obtained by direct Ewald summation 
(Algorithm~\ref{alg:FastEwaldGeneral}). In the reference calculation, we use a sufficiently large number of Fourier modes so that the truncation error is at least one order of magnitude below the target error level. The Fourier-transformed mollifier is evaluated at all available wavenumbers (also outside the bandlimit)
without additional truncation.

Gaussian window functions are defined as in Table~\ref{tab:reference_windows}. For the 
B-spline results, we employ cardinal B-splines of order $P$ together with the standard 
SPME diagonal scaling \cite[Appendix A]{lindboSpectralAccuracyFast2011}.

Errors are reported either as relative $\ell_2$ errors,
\begin{align}
\frac{\|\bm{\phi}_{\mathrm{ref}}^{\mathrm{far}} - \bm{\phi}^{\mathrm{far}}\|_2}
    {\|\bm{\phi}^{\mathrm{far}}_{\mathrm{ref}}\|_2},
\end{align}
or as absolute root-mean-square (RMS) errors,
\begin{align}
\frac{1}{\sqrt{n}}\,
\|\bm{\phi}_{\mathrm{ref}}^{\mathrm{far}} - \bm{\phi}^{\mathrm{far}}\|_2,
\end{align}
where $\|\bm{a}\|_2 = \sqrt{\sum_{j=1}^n |a_j|^2}$ denotes the standard 
$\ell_2$ vector norm. Absolute RMS errors are used when comparing numerical results 
to the theoretical error models in Section~\ref{sec:error_analysis}, which are 
formulated in terms of absolute errors. Relative $\ell_2$ errors are used when comparing different mollifier and window-function configurations.


\subsection{Direct Ewald sum: PSWF vs.~Gaussian split}
\label{sec:Ewald_summation}

To isolate the effect of the mollifier, we compare the PSWF and Gaussian splits in the
direct summation of the far-field potential (Algorithm~\ref{alg:FastEwaldGeneral}),
thereby eliminating approximation errors associated with FFT-based acceleration. For
this test, the shape parameters are expected to satisfy
\begin{align}\label{eq:directsumexample}
    c_s \;\approx\; (r_c/\sigma)^2 \;\approx\; \log(1/\varepsilon),
\end{align}
so that $\psi_0^{c_s}(1)/\psi_0^{c_s}(0) \approx \varepsilon$ for the PSWF split and
$e^{-(r_c/\sigma)^2} \approx \varepsilon$ for the Gaussian split, where $\varepsilon$
denotes the target relative error for the far-field sum.

For the PSWF split, the optimal number of Fourier modes $m$ per dimension is directly
proportional to the bandlimit $c_s$ through the relation $c_s = \omega_{\max} r_c$,
where $\omega_{\max} = \pi m/L$. In the Gaussian case, \eqref{eq:directsumexample}
provides a reasonable choice to choose $\sigma$. We recall that in the PSWF setting there
are no truncation errors in the real-space sum (see
Section~\ref{sec:PSWF_mollifier}), and the split error is therefore entirely determined
by the error of the far-field sum. In contrast, for the Gaussian split, the choice of
$\sigma$ must be based on the truncation error of the real-space sum, and
\eqref{eq:directsumexample} sets $\sigma$ accordingly.

For each $\sigma$, the error decreases with increasing $m$. To measure the resolution
requirements for a given tolerance $\varepsilon$, we therefore compute the relative
$\ell_2$ error for increasing values of $m$ until the measured error falls below
$\varepsilon$, at which point the corresponding value of
$\omega_{\max} = \pi m/L$ is recorded. The results are shown, together with the errors
computed using $c_s = r_c \omega_{\max}$ for the PSWF split, in
Figure~\ref{fig:split_errors}.

\begin{figure}[tbp]
\centering
\includegraphics[width=0.5\textwidth]{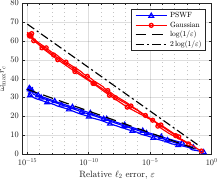}
\caption{Resolution requirements for direct Ewald summation using PSWF and Gaussian
splits. Relative $\ell_2$ errors are shown for cutoff radii
$r_c \in \{0.1, 0.3, 0.5\}$ with $L = 1$ and $n = 100$. Dashed lines indicate reference
slopes $\log(1/\varepsilon)$ (PSWF) and $2\log(1/\varepsilon)$ (Gaussian). Expressing the
resolution as $\omega_{\max} r_c$ causes the results to collapse and confirms the
expected scaling with respect to $r_c$. For practically relevant tolerances, the PSWF
split requires approximately half as many Fourier modes per dimension, corresponding to
an eightfold reduction in three dimensions.}
\label{fig:split_errors}
\end{figure}

Theoretically, by substituting $c_s = r_c\,\omega_{\max}^{\mathrm{PSWF}}$ into
$c_s \approx \log(1/\varepsilon)$, and
$\sigma^2 \approx r_c^2/\log(1/\varepsilon)$ into the Fourier representation of the
Gaussian mollifier,
$e^{-\sigma^2 \omega^2/4} = \varepsilon$ evaluated at
$\omega_{\max}^{\mathrm{Gauss}}$, one obtains the asymptotic relations
\begin{align}
    \omega_{\max}^{\mathrm{PSWF}}
    \sim \frac{\log(1/\varepsilon)}{r_c},
    \qquad
    \omega_{\max}^{\mathrm{Gauss}}
    \sim \frac{2\log(1/\varepsilon)}{r_c}.
\end{align}
Hence, for a given target tolerance,
$\omega_{\max}^{\mathrm{Gauss}} \approx 2\,\omega_{\max}^{\mathrm{PSWF}}$.
Figure~\ref{fig:split_errors} confirms that the empirical error curves follow these
predicted asymptotic trends, implying that the PSWF split requires approximately eight
times fewer Fourier modes than the Gaussian split in three dimensions. The figure also
indicates that the dependence of the relative error on the cutoff radius $r_c$ is weak.

\subsection{Fast Ewald with PSWF split and PSWF window}
\label{sec:PSWF_PSWF_results}

When both the mollifier and the window function are based on PSWFs, the overall 
accuracy depends on two shape parameters: $c_s$, which controls the kernel split, and 
$c_w$, which governs the window function. Proposition~\ref{prop:interpol_est} shows that 
the aliasing error decays rapidly as $c_w$ increases, while 
Proposition~\ref{prop:PSWF_split_errors} characterizes the truncation error associated 
with $c_s$. The total error may be viewed as approximately additive in these two 
contributions (see \eqref{eq:toterror}) and is dominated by the larger of the two.

\begin{figure}[tbp]
\centering
\includegraphics[width=0.95\textwidth]{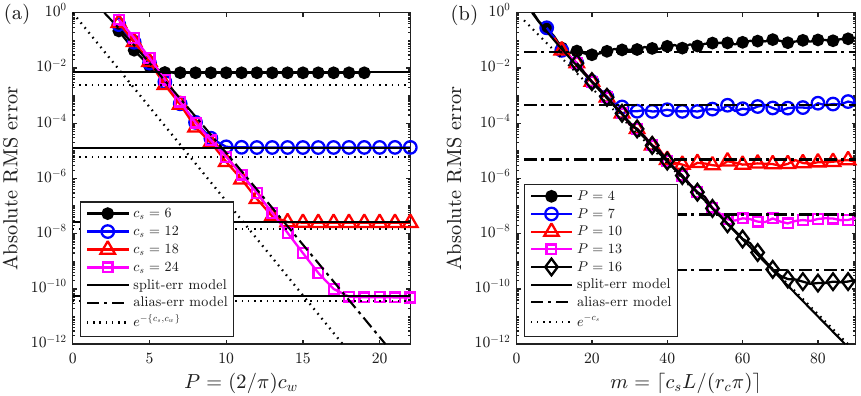}
\caption{Absolute RMS errors for fast Ewald summation 
(Algorithm~\ref{alg:SE_Fourier}) using the PSWF split and PSWF window 
($L = 1$, $r_c = 0.1$, $n = 100$). Curves with symbols show numerical results, while 
black curves indicate theoretical error models. Dotted lines 
indicate the simple asymptotic exponential model $e^{-\{c_s,c_w\}}$; solid lines 
correspond to the split-error model 
$\varepsilon = \tfrac{1}{\sqrt{V}}\sqrt{r_c}\,\|\bm{\rho}\|_2\,A_s c_s^{-1/2} e^{-c_s}$ \eqref{eq:PSWF_model} with $A_s\approx 5$; and 
dash--dotted lines correspond to the aliasing-error model 
$\varepsilon = \tfrac{1}{\sqrt{V}}\sqrt{L}\,\|\bm{\rho}\|_2\,A_w c_w^{1/2} e^{-c_w}$ \eqref{eq:window_err_model} with $A_w\approx 3.1$. 
(a) Absolute RMS error vs.\ window support $P = (2/\pi)c_w$ for fixed $c_s$.
(b) Absolute RMS error vs.\ number of Fourier modes $m$ (equivalently $c_s$) for fixed $P$.}
\label{fig:abserrors}
\end{figure}

Figure~\ref{fig:abserrors}(a) shows that, for fixed $c_s$, the computed absolute RMS error 
decreases as $\mathcal{O}(e^{-c_w})$ until it reaches a plateau determined by the split 
error, which scales as $\mathcal{O}(e^{-c_s})$. The fact that this plateau appears 
perfectly horizontal on a semilogarithmic scale is consistent with the exponential 
decay. 

Conversely, Figure~\ref{fig:abserrors}(b) demonstrates that increasing $c_s$---or, 
equivalently, increasing the number of Fourier modes $m$---reduces the split error until 
the total error saturates at a level $\mathcal{O}(e^{-c_w})$, where 
$c_w = (\pi/2) P$. This saturation level is determined by the aliasing errors introduced 
by the windowing (spreading and interpolation) steps.

We observe that the error curves and plateaus in Figure~\ref{fig:abserrors} are well 
approximated by the error models \eqref{eq:PSWF_model} and \eqref{eq:window_err_model}. 
The theoretical predictions accurately capture both the slopes and the plateau levels. This indicates that the chosen parameter 
values yield computed errors that are close to the prescribed 
tolerance $\varepsilon$.

\begin{figure}[tb]
\centering
\includegraphics[width=0.9\textwidth]{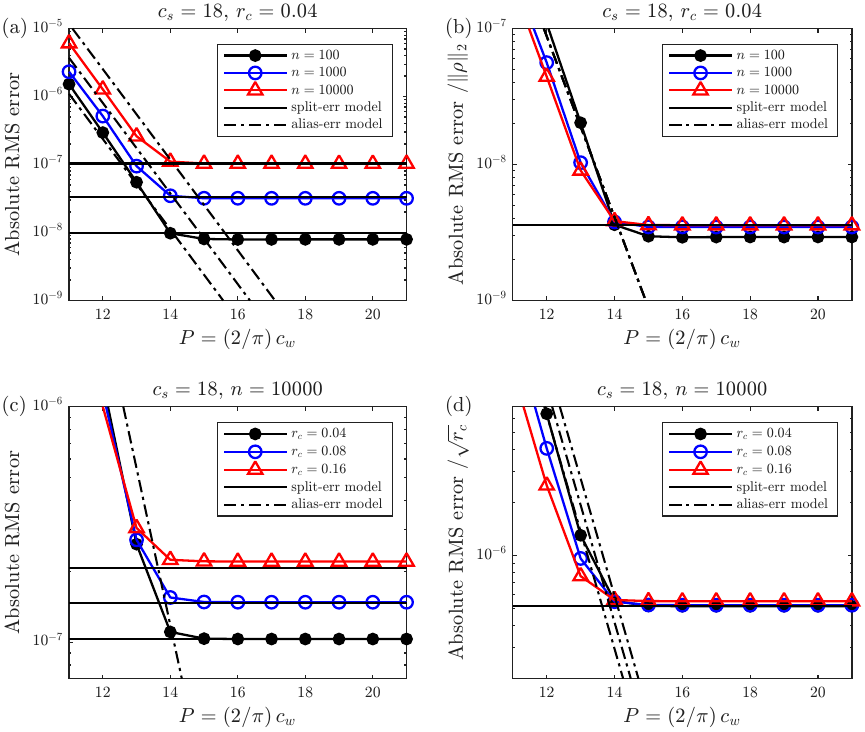}
\caption{Dependence of the absolute RMS error on the system size $n$ and the cutoff 
radius $r_c$ for fast Ewald summation with PSWF split and PSWF mollifier. 
(a) Window support vs.\ absolute RMS error for varying $n$. 
(b) Same data as in (a), rescaled by $\|\bm{\rho}\|_2$. 
(c) Window support vs.\ absolute RMS error for varying $r_c$. 
(d) Same data as in (c), rescaled by $\sqrt{r_c}$. 
Curves with symbols show numerical results, while black curves indicate theoretical error models: 
solid lines correspond to the split-error model 
$\varepsilon = \tfrac{1}{\sqrt{V}}\sqrt{r_c}\,\|\bm{\rho}\|_2\,A_s c_s^{-1/2} e^{-c_s}$ \eqref{eq:PSWF_model}, and 
dash--dotted lines to the aliasing-error model 
$\varepsilon = \tfrac{1}{\sqrt{V}}\sqrt{L}\,\|\bm{\rho}\|_2\,A_w c_w^{1/2} e^{-c_w}$ \eqref{eq:window_err_model}. 
}
\label{fig:n_rc_dependency}
\end{figure}
\begin{figure}[tb]
\centering
\vspace{5pt}
\includegraphics[width=0.95\textwidth]{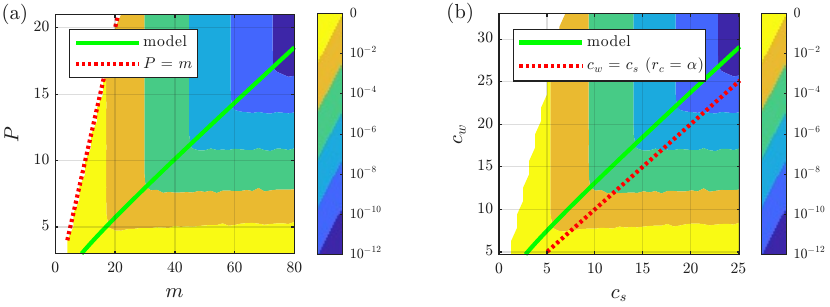} 
\caption{Contour plots of the absolute RMS error ($L = 1$, $r_c = 0.1$, $n = 100$). 
In (a), error in $(m,P)$ space. In (b), error in $(c_s,c_w)$ space. The two parameter spaces are 
related linearly by $c_s = (\pi r_c/L)\,m$ and $c_w = (\pi/2)\,P$. The green curve in each 
panel shows the theoretical optimal relation~\eqref{eq:cwcs_model_dimless}; the computed 
minimizers $(m,P)$ (or equivalently $(c_s,c_w)$) for each error level align closely with 
this curve. For reference, the curve $  P = m$ is also shown, indicating a simple but 
non-optimal parameter choice, as well as the curve $c_w = c_s$, corresponding to the 
case $r_c = \alpha$, which overestimates $c_s$ for a given error level.}
 \label{fig:PSWF_contour} 
\end{figure}

Figures~\ref{fig:n_rc_dependency}(a)--(b) confirm that the absolute RMS error grows  
linearly with $\|\bm{\rho}\|_2$. After rescaling by $\|\bm{\rho}\|_2$, the 
error curves collapse, in agreement with the linearity of the potential. 

Figures~\ref{fig:n_rc_dependency}(c)--(d) further show that the truncation error scales as 
$\sqrt{r_c}$, in agreement with Proposition~\ref{prop:PSWF_split_errors}. Taken 
together, the results in Figure~\ref{fig:n_rc_dependency} confirm that the PSWF-based 
formulation exhibits the expected scaling behavior with respect to both the charge 
amplitude and the cutoff radius.


Figure~\ref{fig:PSWF_contour} shows the absolute RMS error over a grid of 
$(m,P)$ values. Since $P = (2/\pi)\,c_w$ and $c_s = (\pi r_c/L)\,m$, the mapping 
$(m,P) \mapsto (c_s,c_w)$ is linear, and the right panel therefore displays the same data 
as the left panel, but expressed in $(c_s,c_w)$ coordinates. The plots demonstrate that 
the model \eqref{eq:cwcs_model_dimless} is in sharp agreement with the minimal $(m,P)$ 
(or equivalently $(c_s,c_w)$) values corresponding to each error level. The figure also includes the curves $P = m$ and $c_w = c_s$, which represent simple but 
non-optimal choices for parameter selection.

Using the parameter selection procedure described in Algorithm~\ref{alg:parameter}, we aim to achieve a prescribed target accuracy for a given tolerance $\varepsilon$.
Figure~\ref{fig:errorvstol} shows the measured absolute RMS error plotted as functions of $\varepsilon$.
We observe excellent agreement between the prescribed tolerance and the measured error, indicating that the parameter selection procedure performs as intended.

\begin{figure}[tbp]
\centering
\includegraphics[width=0.9\textwidth]{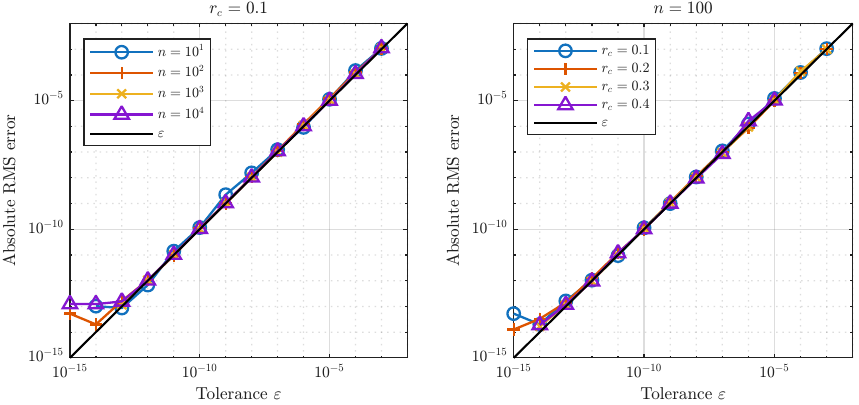}
\caption{Measured absolute RMS error as a function of the requested tolerance $\varepsilon$ for different values of $n$ and $r_c$.}
\label{fig:errorvstol}
\end{figure}

\subsubsection{Comparison of PSWF/PSWF, Gaussian/PSWF, and Gaussian/Gaussian configurations}
\label{sec:pswf_comparisons}

Next, we compare the errors of the PSWF/PSWF (split/window) configuration against the Gaussian/PSWF and Gaussian/Gaussian cases. All three are spectrally accurate Ewald methods with exponential convergence of the error in $m$ and $P$. We treat PSWF/PSWF as our main reference; the Gaussian/PSWF comparison demonstrates the benefit of the PSWF split, while the Gaussian/Gaussian comparison provides context against the original spectral Ewald \cite{lindboSpectralAccuracyFast2011} configuration.

\begin{table}[tbp]
\centering
\caption{Resolution requirements to achieve a prescribed relative $\ell_2$ error for Gaussian (G) and PSWF (P) splits in standard (direct) Ewald summation, and for Gaussian/Gaussian (GG), Gaussian/PSWF (GP), and PSWF/PSWF (PP) configurations in fast Ewald summation.
 The table reports the number of Fourier modes 
$m$, the window support $P$, and the cubic ratios $(m^{X}/m^{\mathrm{PP}})^3$, where 
$X \in \{\mathrm{G}, \mathrm{P}, \mathrm{GG}, \mathrm{GP}, \mathrm{PP}\}$. The shape 
parameters were chosen as 
$c_s = (r_c/\sigma)^2 = \log(1/\varepsilon)$. All results correspond to 
$L = 1$, $r_c = 0.1$, and $n = 100$.}

\label{tab:resolution1}
\footnotesize
\begin{tabularx}{\textwidth}{c|c|c|c|c|c|c|c|c|C}
\toprule
$\varepsilon$ & $\log(\varepsilon^{-1})$ & $m^{\mathrm{G}}$ & $m^{\mathrm{P}}$ &
$\big(\tfrac{m^{\mathrm{G}}}{m^{\mathrm{P}}}\big)^3$ &
$m^{\mathrm{GG}}\!, P^{\mathrm{GG}}$ &
$m^{\mathrm{GP}}\!, P^{\mathrm{GP}}$ &
$m^{\mathrm{PP}}\!, P^{\mathrm{PP}}$ &
$\big(\tfrac{m^{\mathrm{GG}}}{m^{\mathrm{PP}}}\big)^3$ &
$\big(\tfrac{P^{\mathrm{GG}}}{P^{\mathrm{PP}}}\big)^3$ \\
\midrule
$10^{-2}$   & 4.61 & 22  & 13  & 4.8 & \hphantom{1}24, 6\hphantom{0}   & \hphantom{1}24, 5\hphantom{1}   & 13, 5\hphantom{1}  & 6.3 & 1.7\\
$10^{-3}$   & 6.91 & 35  & 20  & 5.4 & \hphantom{1}37, 9\hphantom{0}   & \hphantom{1}35, 6\hphantom{1}   & 20, 6\hphantom{1}  & 6.3 & 3.4\\
$10^{-4}$   & 9.21 & 48  & 27  & 5.6 & \hphantom{1}51, 11  & \hphantom{1}52, 8\hphantom{1}   & 27, 8\hphantom{1}  & 6.7 & 2.6\\
$10^{-5}$   & 11.51 & 62  & 35  & 5.6 & \hphantom{1}65, 13  & \hphantom{1}66, 9\hphantom{1}   & 35, 9\hphantom{1}  & 6.7 & 3.0\\
$10^{-6}$   & 13.82 & 75  & 42  & 5.7 & \hphantom{1}80, 15  & \hphantom{1}79, 10  & 42, 10 & 6.9 & 3.4\\
$10^{-7}$   & 16.12 & 89  & 49  & 6.0 & \hphantom{1}94, 19  & \hphantom{1}94, 12  & 49, 12 & 7.1 & 4.0\\
$10^{-8}$   & 18.42 & 103 & 57  & 5.9 & 109, 20 & 109, 13 & 57, 13 & 7.0 & 3.6\\
$10^{-9}$   & 20.72 & 117 & 64  & 6.1 & 123, 22 & 123, 15 & 64, 15 & 7.1 & 3.2\\
$10^{-10}$  & 23.03 & 131 & 72  & 6.0 & 138, 24 & 138, 16 & 72, 16 & 7.0 & 3.4\\
$10^{-11}$  & 25.33 & 146 & 79  & 6.3 & 154, 25 & 154, 17 & 79, 17 & 7.4 & 3.2\\
$10^{-12}$  & 27.63 & 170 & 86  & 7.7 & 172, 26 & 172, 18 & 86, 18 & 8.0 & 3.0\\
\bottomrule
\end{tabularx}
\end{table}

Table~\ref{tab:resolution1} summarizes the minimum number of Fourier modes required by each configuration to achieve prescribed target accuracies. We choose
$c_s = (r_c/\sigma)^{2} = \log(1/\varepsilon)$ in order to match the truncation levels
for the Gaussian and PSWF functions. We remark that this parameter choice is not optimal for achieving the prescribed error tolerance without additional tuning. However, for the present test, where the goal is simply to select a reasonable shape parameter for comparison with the Gaussian case, it is perfectly adequate. In the table, we report, for each tolerance level $\varepsilon$, the shape parameters used, the smallest number of Fourier modes~$m$, and the minimal window support~$P$ that satisfy the prescribed
tolerance.

For the PSWF/PSWF case, the window support values~$P$ were chosen as the smallest values for which the error corresponding to a given $c_s$ fell below the prescribed tolerance~$\varepsilon$. This selection can be understood from Figure~\ref{fig:abserrors}(a) by following the error curves from left to right. Owing to the presence of well-defined error plateaus, the appropriate $P$ values in the PSWF/PSWF case are straightforward to identify. Moreover, since the same functional form governs the decay of both the split error and the windowing error, the number of Fourier modes required to achieve a given tolerance is comparable to that required by the direct Fourier sum.

For the Gaussian/Gaussian case, determining the minimal $(m,P)$ combination is more subtle. It is well known that this algorithm does not exhibit sharply defined error plateaus, due to the nature of its approximation mechanisms, for example the use of non-truncated Gaussian functions. Consequently, the error curves transition smoothly from the exponential decay regime to the asymptotic error level, rather than displaying 
a clear saturation plateau. In this context, 
\cite{baggeFastEwaldSummation2023,shamshirgarFastEwaldSummation2021} recommend choosing 
the parameters $m$ and $P$ slightly larger than what idealized error models would 
suggest. In particular, it is suggested to increase $m$ by a factor of $1.05$, a 
strategy that we also adopt here. This explains why the resulting $m$ values for the 
Gaussian/Gaussian case are slightly larger---approximately $5\%$, up to rounding---than 
the corresponding values obtained using the direct Gaussian sum reported in the table.

Table~\ref{tab:resolution1} shows that the PSWF/PSWF method achieves the prescribed
accuracy using both fewer Fourier modes and a smaller window support than the
Gaussian/Gaussian method. In the large-$c_s$ regime, the Gaussian split requires
approximately eight times as many Fourier modes as the PSWF split to reach the target
tolerance for the present configuration. As the tolerance is tightened, this ratio
increases toward its asymptotic value. Taking into account the $1.05$ correction factor
applied to the Gaussian/Gaussian case, we estimate that the upper bound on the ratio of
required Fourier modes between the Gaussian/Gaussian and PSWF/PSWF methods is
approximately $8 \times 1.05^3 \approx 9.3$. In
Table~\ref{tab:resolution1}, the observed ratio reaches $8.0$ at
$\varepsilon = 10^{-12}$; however, this value does not yet reflect the asymptotic limit.
In addition, the required window support~$P$ is reduced by approximately a factor of
three when using the PSWF window compared to the Gaussian window.

The results for the mixed Gaussian/PSWF configuration are also included in 
Table~\ref{tab:resolution1}. As expected, the required Fourier resolution is comparable 
to that of the Gaussian/Gaussian case, while the window support is determined by—and 
matches—that of the PSWF window.

\subsection{Fast Ewald with PSWF split and B-spline window}
\label{sec:PSWF_Bspline}

In both the PSWF/PSWF and Gaussian/Gaussian cases, the split error and the window
(aliasing) error decay exponentially fast. For the PSWF/PSWF case, we observe in
Figure~\ref{fig:abserrors}(b) that, for fixed window support~$P$, the total error decays
as $\mathcal{O}(e^{-m})$ until it reaches the windowing error level determined by $P$. On a semilogarithmic scale, this behavior appears as a straight
line with negative slope, which levels off into a horizontal plateau once the
windowing error dominates. Although aliasing errors are formally present also for small
values of~$m$, the total error in this regime is dominated by the split truncation error. For fixed window support~$P$, the windowing (aliasing) error is
therefore essentially independent of~$m$ and determines an error floor that cannot be
reduced by increasing the Fourier grid size. The same qualitative behavior is observed
in the Gaussian/Gaussian case~\cite{shamshirgarFastEwaldSummation2021}.

By contrast, when a B-spline window is used, the aliasing errors introduced by the
windowing decay more slowly than the split errors, namely only algebraically. For fixed
spline order (window support)~$P$, the decay is typically of the form
$\mathcal{O}(m^{-P})$ \cite{schoenbergCardinalSplineInterpolation1973}. As a result, the total error is eventually dominated by the
windowing contribution, and, depending on the prefactor, this may occur already for
relatively small values of~$m$.

\begin{figure}[tbp]
\centering
\includegraphics{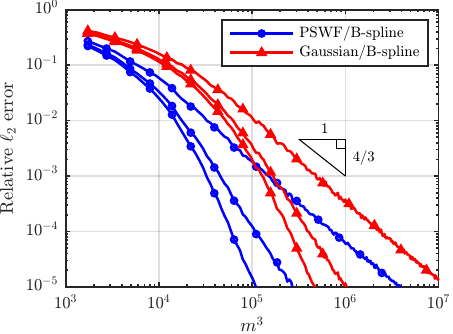}
\caption{Total number of Fourier modes $m^3$ versus relative $\ell_2$ error for PSWF/B-spline and Gaussian/B-spline (SPME) configurations at $L=1$, $r_c=0.1$, $n=100$, and B-spline order $P\in\{4,6,8\}$. The split parameters $c_s$ (PSWF) and $\sigma$ (Gaussian) were chosen to yield the same direct-sum split error of $10^{-5}$. For small and moderate $m$, the PSWF split yields faster error decay and reaches a given tolerance with fewer modes; for large $m$, both curves exhibit the same algebraic slope governed by the B-spline window. The dotted horizontal line indicates an error of $10^{-5}$; the corresponding $m$-values where the curves cross this line are listed in Table~\ref{tab:resolution2}.}
\label{fig:B-spline}
\end{figure}

This effect is clearly illustrated in Figure~\ref{fig:B-spline}, which shows the relative
$\ell_2$ error for the PSWF/B-spline and Gaussian/B-spline configurations. In these experiments, the
split parameters are chosen as
$c_s = (r_c/\sigma)^2 = \log(1/\varepsilon)$ with $\varepsilon = 10^{-5}$, and B-spline
orders (which coincide with the window support) $P \in \{4,6,8\}$. 

In Figure~\ref{fig:B-spline}, we observe that the error curves corresponding to the
Gaussian/B-spline configuration are shifted upward relative to those of the
PSWF/B-spline method. This shift is a direct consequence of the faster decay of the PSWF
split error. For smaller values of $m^3$, the split error contributes significantly to
the total error. As $m^3$ is increased, the exponentially decaying split error becomes
negligible, and the total error is instead governed by the algebraically decaying
window-induced aliasing error.

To verify the algebraic convergence rate, a reference triangle is included in the plot
for the case $P = 4$, indicating the expected slope on a log--log scale. The cases
$P = 6$ and $P = 8$ exhibit consistent slopes, although reference triangles are omitted
for clarity.

In all B-spline test cases considered here, the shape parameters were first chosen
approximately at the split-error level using $\log(1/\varepsilon)$, so that the split
error was comparable to the target tolerance. A high-resolution reference solution was
then computed using direct PSWF or Gaussian summation. The error was subsequently
evaluated for increasing Fourier grid sizes~$m$ until the prescribed tolerance
$\varepsilon$ was reached, at which point the corresponding value of~$m$ was recorded.
One may question whether this procedure provides an appropriate basis for comparing
resolution requirements, since, in the PSWF split case, the prescribed tolerance
$\varepsilon$ is often reached at values of~$m$ that are significantly smaller than those required to reduce the B-spline aliasing error to the same level. However, this behavior is precisely what motivates the chosen procedure. Increasing the
PSWF bandlimit parameter~$c_s$ increases the number of Fourier modes required to reach a
given tolerance, which manifests itself as an upward shift of the error curves. As a
result, the minimal value of~$m$ in the PSWF/B-spline case is attained when $c_s$ is
chosen such that the split truncation error is $\mathcal{O}(\varepsilon)$, thereby
minimizing the Fourier grid size for a fixed window support. Choosing a smaller~$c_s$
causes the split error to decay rapidly with increasing~$m$ until slowly decaying
$\mathrm{sinc}$-type oscillations in the PSWF spectrum dominate the total error (see
Figure~\ref{fig:PSWFvsGaussian}). In this regime, further reductions in~$m$ are not
possible without increasing the window support. Consequently, selecting $c_s$ so that the split error is comparable to the target
tolerance yields the smallest Fourier grid size compatible with the prescribed accuracy
for a fixed B-spline order, making the above procedure appropriate for assessing
resolution requirements. The Gaussian (nontruncated) case is somewhat more subtle;
nevertheless, we expect qualitatively similar behavior.

\begin{table}[tbp]
\centering
\caption{Fourier-mode requirements to achieve target tolerances 
$\varepsilon \in \{10^{-2}, 10^{-3}, \ldots, 10^{-6}\}$ (relative $\ell_2$ error) for PSWF/B-spline and 
Gaussian/B-spline (SPME) configurations with $L = 1$, $r_c = 0.1$, $n = 100$, and 
B-spline orders $P \in \{4,6,8\}$ (as in Figure~\ref{fig:B-spline}). The table reports the 
minimum number of Fourier modes $m$, the cubic ratios 
$(m^{\mathrm{G}}_P / m^{\mathrm{P}}_P)^3$ between the Gaussian/B-spline (G) and 
PSWF/B-spline (P) cases, and the overall ratio $(m^{\mathrm{G}}_P / m^{\mathrm{PP}})^3$ 
relative to the PSWF/PSWF (PP) reference.}
\label{tab:resolution2}
\footnotesize
\begin{tabularx}{\textwidth}{@{}c|C|C|C|C@{}}
\toprule
$\varepsilon$ &
\makebox[0pt][c]{\begin{tabular}{@{}c@{}}$m_P^{\mathrm{G}}$\\[2pt]$(P{=}4,6,8)$\end{tabular}} &
\makebox[0pt][c]{\begin{tabular}{@{}c@{}}$m_P^{\mathrm{P}}$\\[2pt]$(P{=}4,6,8)$\end{tabular}} &
\makebox[0pt][c]{\begin{tabular}{@{}c@{}}$(m_P^{\mathrm{G}}/m_P^{\mathrm{P}})^3$\\[2pt]$(P{=}4,6,8)$\end{tabular}} &
\makebox[0pt][c]{\begin{tabular}{@{}c@{}}$(m_P^{\mathrm{G}}/m^{\mathrm{PP}})^3$\\[2pt]$(P{=}4,6,8)$\end{tabular}} \\
\midrule
$10^{-2}$ & \hphantom{1}31, \hphantom{1}25, 24\hphantom{1} & \hphantom{1}18, \hphantom{1}15, 14 & 5.11, 4.63, 5.04 & \hphantom{2}13.56, \hphantom{1}7.11, 6.29\hphantom{1} \\
$10^{-3}$ & \hphantom{1}63, \hphantom{1}43, 39\hphantom{1} & \hphantom{1}39, \hphantom{1}26, 23 & 4.22, 4.52, 4.88 & \hphantom{2}31.26, \hphantom{1}9.94, 7.41\hphantom{1} \\
$10^{-4}$ & 117, \hphantom{1}67, 57\hphantom{1} & \hphantom{1}76, \hphantom{1}42, 34 & 3.65, 4.06, 4.71 & \hphantom{2}81.37, 15.28, 9.41\hphantom{1} \\
$10^{-5}$ & 230, 101, 78\hphantom{1} & 154, \hphantom{1}66, 49 & 3.33, 3.58, 4.03 & 283.78, 24.03, 11.07 \\
$10^{-6}$ & 455, 156, 106 & 308, 100, 66 & 3.22, 3.80, 4.14 & 1271.4, 51.24, 16.08 \\
\bottomrule
\end{tabularx}

\end{table}

Table~\ref{tab:resolution2} summarizes the minimal values of $m$ required to reach the 
tolerances $\varepsilon \in \{10^{-2},10^{-3},\ldots,10^{-6}\}$. In practice, error 
tolerances in this range are typical (see, e.g., 
GROMACS~\cite{vanderspoelGROMACSFastFlexible2005}). For smaller tolerances, the use of 
B-spline windows becomes prohibitively expensive compared to the PSWF/PSWF and 
Gaussian/Gaussian configurations, and these cases are therefore omitted.

The results show that employing a PSWF split provides a clear advantage when 
B-spline windows are required by the implementation. Depending on the tolerance $\varepsilon$ the 
number of Fourier modes is reduced by a factor of approximately $3$–$5$ compared to 
the Gaussian/B-spline case, with larger improvements observed for smaller tolerances. 
In contrast, the Gaussian/B-spline configuration is substantially less efficient than 
the PSWF/PSWF reference. For example, for $P = 4$ and $\varepsilon = 10^{-6}$, the 
PSWF/PSWF method requires roughly three orders of magnitude fewer Fourier modes than 
the Gaussian/B-spline method, while even for $P = 8$ a reduction by about one order of 
magnitude is observed at the practically relevant tolerance $\varepsilon = 10^{-5}$.

\section{Conclusions}
\label{sec:conclusions}

We have presented a complete description of fast Ewald summation based on the first prolate spheroidal wave function of order zero (PSWF). While PSWFs have recently attracted attention in this context, our main contribution is a systematic and unified treatment: we define PSWF-based mollifier and window functions, establish their analytical properties, and integrate them consistently into the fast Ewald framework.

The resulting PSWF-based split has several advantageous properties. Because PSWFs form an exact Fourier transform pair on a bandlimited interval, the Fourier-space representation is particularly simple, and the residual kernel becomes compactly supported, eliminating truncation errors in the real-space sum. Consequently, the total approximation error is governed by Fourier truncation and aliasing. Owing to the superior spectral localization of PSWFs, a prescribed accuracy can be reached with substantially fewer Fourier modes and smaller window supports than in Gaussian- and B-spline-based methods. Numerical experiments confirm these predictions and demonstrate significant reductions in computational cost at fixed accuracy.

We have derived a detailed error analysis for both Fourier truncation and window-induced aliasing. The resulting error estimates and simplified closed-form models exhibit exponential decay and enable transparent and reliable parameter selection, providing a practical framework for balancing accuracy and computational cost.

Future work includes extending the PSWF-Ewald method and the error analysis to other interaction kernels, such as those arising in Stokes flow and wave propagation, and investigating the potential benefits of oversampling to further accelerate fast Ewald summation.

%
%
%
%

\bibliographystyle{abbrv}
\bibliography{zoteroexport.bib}

@Article{LAMMPS,
  author = "A. P. Thompson and H. M. Aktulga and R. Berger and 
     D. S. Bolintineanu and W. M. Brown and P. S. Crozier and
     P. J. in 't Veld and A. Kohlmeyer and S. G. Moore and T. D. Nguyen and
     R. Shan and M. J. Stevens and J. Tranchida and C. Trott and S. J. Plimpton",
  title = "{LAMMPS} - a flexible simulation tool for
     particle-based materials modeling at the 
     atomic, meso, and continuum scales",
  journal = "Computer Physics Communications",
  volume =  "271",
  pages =   "108171",
  year =    "2022",
  doi = "10.1016/j.cpc.2021.108171"
}

@article{afklintebergFastEwaldSummation2017,
  title = {Fast {{Ewald}} Summation for Free-Space {{Stokes}} Potentials},
  author = {{af Klinteberg}, Ludvig and Shamshirgar, Davoud Saffar and Tornberg, Anna-Karin},
  year = 2017,

  journal = {Research in the Mathematical Sciences},
  volume = {4},
  number = {1},
  pages = {1},
  issn = {2197-9847},
  doi = {10.1186/s40687-016-0092-7},
  urldate = {2024-10-05},
  langid = {english}
}

@article{baggeFastEwaldSummation2023,
  title = {Fast {{Ewald}} Summation for {{Stokes}} Flow with Arbitrary Periodicity},
  author = {Bagge, Joar and Tornberg, Anna-Karin},
  year = 2023,

  journal = {Journal of Computational Physics},
  volume = {493},
  pages = {112473},
  issn = {00219991},
  doi = {10.1016/j.jcp.2023.112473},
  urldate = {2024-10-05},
  langid = {english}
}

@article{barnettAliasingError$expbeta2021,
  title = {Aliasing error of the $\exp(\beta \sqrt{1-z^2})$ kernel in the nonuniform fast Fourier transform},
  author = {Barnett, Alex H.},
  year = 2021,

  journal = {Applied and Computational Harmonic Analysis},
  volume = {51},
  pages = {1-16},
  issn = {10635203},
  doi = {10.1016/j.acha.2020.10.002},
  urldate = {2024-12-28},
  langid = {english}
}

@article{barnettParallelNonuniformFast2019a,
  title = {A {{Parallel Nonuniform Fast Fourier Transform Library Based}} on an ``{{Exponential}} of {{Semicircle}}" {{Kernel}}},
  author = {Barnett, Alexander H. and Magland, Jeremy and Af Klinteberg, Ludvig},
  year = 2019,

  journal = {SIAM Journal on Scientific Computing},
  volume = {41},
  number = {5},
  pages = {C479-C504},
  issn = {1064-8275, 1095-7197},
  doi = {10.1137/18M120885X},
  urldate = {2024-10-05},
  langid = {english}
}

@article{beylkinFastFourierTransform1995,
  title = {On the {{Fast Fourier Transform}} of {{Functions}} with {{Singularities}}},
  author = {Beylkin, G.},
  year = 1995,

  journal = {Applied and Computational Harmonic Analysis},
  volume = {2},
  number = {4},
  pages = {363--381},
  issn = {1063-5203},
  doi = {10.1006/acha.1995.1026},
  urldate = {2025-05-08}
}

@book{chebfun,
  title = {Chebfun {{Guide}}},
  author = {{T. A. Driscoll, N. Hale, and L. N. Trefethen, editors}},
  year = 2014,
  publisher = {Pafnuty Publications},
  address = {Oxford},
  urldate = {2025-05-13}
}

@article{cooleyAlgorithmMachineCalculation1965,
  title = {An {{Algorithm}} for the {{Machine Calculation}} of {{Complex Fourier Series}}},
  author = {Cooley, James W. and Tukey, John W.},
  year = 1965,
  journal = {Mathematics of Computation},
  volume = {19},
  number = {90},
  eprint = {2003354},
  eprinttype = {jstor},
  pages = {297--301},
  publisher = {American Mathematical Society},
  issn = {0025-5718},
  doi = {10.2307/2003354},
  urldate = {2025-12-04}
}

@article{dardenParticleMeshEwald1993,
  title = {Particle Mesh {{Ewald}}: {{An}} {{{\emph{N}}}} {$\cdot$}log( {{{\emph{N}}}} ) Method for {{Ewald}} Sums in Large Systems},
  shorttitle = {Particle Mesh {{Ewald}}},
  author = {Darden, Tom and York, Darrin and Pedersen, Lee},
  year = 1993,

  journal = {The Journal of Chemical Physics},
  volume = {98},
  number = {12},
  pages = {10089--10092},
  issn = {0021-9606, 1089-7690},
  doi = {10.1063/1.464397},
  urldate = {2025-03-08},
  langid = {english}
}

@article{dunsterUniformAsymptoticExpansions1986,
  title = {Uniform {{Asymptotic Expansions}} for {{Prolate Spheroidal Functions}} with {{Large Parameters}}},
  author = {Dunster, T. M.},
  year = 1986,

  journal = {SIAM Journal on Mathematical Analysis},
  volume = {17},
  number = {6},
  pages = {1495--1524},
  issn = {0036-1410, 1095-7154},
  doi = {10.1137/0517108},
  urldate = {2025-01-02},
  langid = {english}
}

@article{duttFastFourierTransforms1993a,
  title = {Fast {{Fourier Transforms}} for {{Nonequispaced Data}}},
  author = {Dutt, A. and Rokhlin, V.},
  year = 1993,

  journal = {SIAM Journal on Scientific Computing},
  volume = {14},
  number = {6},
  pages = {1368--1393},
  issn = {1064-8275, 1095-7197},
  doi = {10.1137/0914081},
  urldate = {2025-06-01},
  langid = {english}
}

@article{essmannSmoothParticleMesh1995,
  title = {A Smooth Particle Mesh {{Ewald}} Method},
  author = {Essmann, Ulrich and Perera, Lalith and Berkowitz, Max L. and Darden, Tom and Lee, Hsing and Pedersen, Lee G.},
  year = 1995,

  journal = {The Journal of Chemical Physics},
  volume = {103},
  number = {19},
  pages = {8577--8593},
  issn = {0021-9606, 1089-7690},
  doi = {10.1063/1.470117},
  urldate = {2024-12-10},
  langid = {english}
}

@article{ewaldBerechnungOptischerUnd1921,
  title = {Die {{Berechnung}} Optischer Und Elektrostatischer {{Gitterpotentiale}}},
  author = {Ewald, P. P.},
  year = 1921,
  journal = {Annalen der Physik},
  volume = {369},
  number = {3},
  pages = {253--287},
  issn = {1521-3889},
  doi = {10.1002/andp.19213690304},
  urldate = {2024-12-29},
  copyright = {Copyright \copyright{} 1921 WILEY-VCH Verlag GmbH \& Co. KGaA, Weinheim},
  langid = {english}
}

@article{fuchsEigenvaluesIntegralEquation1964,
  title = {On the Eigenvalues of an Integral Equation Arising in the Theory of Band-Limited Signals},
  author = {Fuchs, W. H. J},
  year = 1964,

  journal = {Journal of Mathematical Analysis and Applications},
  volume = {9},
  number = {3},
  pages = {317--330},
  issn = {0022-247X},
  doi = {10.1016/0022-247X(64)90017-4},
  urldate = {2025-12-11}
}

@article{greengardAcceleratingNonuniformFast2004,
  title = {Accelerating the {{Nonuniform Fast Fourier Transform}}},
  author = {Greengard, Leslie and Lee, June-Yub},
  year = 2004,
  journal = {SIAM Review},
  volume = {46},
  number = {3},
  eprint = {20453531},
  eprinttype = {jstor},
  pages = {443--454},
  publisher = {{Society for Industrial and Applied Mathematics}},
  issn = {0036-1445},
  urldate = {2024-12-29}
}

@book{hockneyComputerSimulationUsing2021,
  title = {Computer {{Simulation Using Particles}}},
  author = {Hockney, R. W. and Eastwood, J. W.},
  year = 2021,

  publisher = {CRC Press},
  address = {Boca Raton},
  doi = {10.1201/9780367806934},
  isbn = {978-0-367-80693-4}
}

@article{hofmannNFFTBasedEwald2017a,
  title = {{{NFFT}} Based {{Ewald}} Summation for Electrostatic Systems with Charges and Dipoles},
  author = {Hofmann, Michael and Nestler, Franziska and Pippig, Michael},
  year = 2017,

  journal = {Applied Numerical Mathematics},
  volume = {122},
  pages = {39--65},
  issn = {01689274},
  doi = {10.1016/j.apnum.2017.07.008},
  urldate = {2025-05-23},
  langid = {english}
}

@article{jiangDualspaceMultilevelKernelsplitting,
  title = {A Dual-Space Multilevel Kernel-Splitting Framework for Discrete and Continuous Convolution},
  author = {Jiang, Shidong and Greengard, Leslie},
  journal = {Communications on Pure and Applied Mathematics},
  volume={78},
  number={5},
  pages={1086--1143},
  year={2025},
  doi = {10.1002/cpa.22240},
  urldate = {2024-12-29},
  chapter = {78},
  copyright = {\copyright{} 2024 Wiley Periodicals LLC.},
  langid = {english}
}

@article{keiner2009using,
  title = {Using {{NFFT}} 3---{{A}} Software Library for Various Nonequispaced Fast {{Fourier}} Transforms},
  author = {Keiner, Jens and Kunis, Stefan and Potts, Daniel},
  year = 2009,
  journal = {ACM Transactions on Mathematical Software (TOMS)},
  volume = {36},
  number = {4},
  pages = {1--30},
  publisher = {ACM New York, NY, USA}
}

@article{klintebergFastSummationStokes2025,
  title = {Fast Summation of {{Stokes}} Potentials Using a New Kernel-Splitting in the {{DMK}} Framework},
  author = {af Klinteberg, Ludvig and Greengard, Leslie and Jiang, Shidong and Tornberg, Anna-Karin},
  year = 2025,

  journal = {arXiv:2509.21471 [math.NA]},
  eprint = {2509.21471},
  primaryclass = {math},
  publisher = {arXiv},
  doi = {10.48550/arXiv.2509.21471},
  urldate = {2026-02-05},
  archiveprefix = {arXiv}
}

@article{liangAcceleratingFastEwald2025,
  title = {Accelerating {{Fast Ewald Summation}} with {{Prolates}} for {{Molecular Dynamics Simulations}}},
  author = {Liang, Jiuyang and Lu, Libin and Barnett, Alex and Greengard, Leslie and Jiang, Shidong},
  year = 2025,
  journal = {arXiv:2505.09727 [math.NA]},
  eprint = {2505.09727},
  primaryclass = {math},
  publisher = {arXiv},
  doi = {10.48550/arXiv.2505.09727},
  urldate = {2025-05-23},
  archiveprefix = {arXiv},
  langid = {english}
}

@article{lindboSpectralAccuracyFast2011,
  title = {Spectral Accuracy in Fast {{Ewald-based}} Methods for Particle Simulations},
  author = {Lindbo, Dag and Tornberg, Anna-Karin},
  year = 2011,

  journal = {Journal of Computational Physics},
  volume = {230},
  number = {24},
  pages = {8744--8761},
  issn = {00219991},
  doi = {10.1016/j.jcp.2011.08.022},
  urldate = {2024-10-05},
  copyright = {https://www.elsevier.com/tdm/userlicense/1.0/},
  langid = {english}
}

@article{nestlerAutomatedParameterTuning2016,
  title = {Automated Parameter Tuning Based on {{RMS}} Errors for Nonequispaced {{FFTs}}},
  author = {Nestler, Franziska},
  year = 2016,

  journal = {Advances in Computational Mathematics},
  volume = {42},
  number = {4},
  pages = {889--919},
  issn = {1019-7168, 1572-9044},
  doi = {10.1007/s10444-015-9446-8},
  urldate = {2025-04-14},
  langid = {english}
}

@article{nestlerFastEwaldSummation2015,
  title = {Fast {{Ewald}} Summation Based on {{NFFT}} with Mixed Periodicity},
  author = {Nestler, Franziska and Pippig, Michael and Potts, Daniel},
  year = 2015,

  journal = {Journal of Computational Physics},
  volume = {285},
  pages = {280--315},
  issn = {00219991},
  doi = {10.1016/j.jcp.2014.12.052},
  urldate = {2025-01-19},
  langid = {english}
}

@article{osipovEvaluationProlateSpheroidal2014,
  title = {On the Evaluation of Prolate Spheroidal Wave Functions and Associated Quadrature Rules},
  author = {Osipov, Andrei and Rokhlin, Vladimir},
  year = 2014,

  journal = {Applied and Computational Harmonic Analysis},
  volume = {36},
  number = {1},
  pages = {108--142},
  issn = {10635203},
  doi = {10.1016/j.acha.2013.04.002},
  urldate = {2024-12-28},
  langid = {english}
}

@book{osipovProlateSpheroidalWave2013,
  title = {Prolate {{Spheroidal Wave Functions}} of {{Order Zero}}: {{Mathematical Tools}} for {{Bandlimited Approximation}}},
  shorttitle = {Prolate {{Spheroidal Wave Functions}} of {{Order Zero}}},
  author = {Osipov, Andrei and Rokhlin, Vladimir and Xiao, Hong},
  year = 2013,
  series = {Applied {{Mathematical Sciences}}},
  volume = {187},
  publisher = {Springer US},
  address = {Boston, MA},
  doi = {10.1007/978-1-4614-8259-8},
  urldate = {2024-12-31},
  copyright = {https://www.springernature.com/gp/researchers/text-and-data-mining},
  isbn = {978-1-4614-8258-1 978-1-4614-8259-8},
  langid = {english}
}

@book{plonkaNumericalFourierAnalysis2018,
  title = {Numerical {{Fourier Analysis}}},
  author = {Plonka, Gerlind and Potts, Daniel and Steidl, Gabriele and Tasche, Manfred},
  year = 2018,
  series = {Applied and {{Numerical Harmonic Analysis}}},
  publisher = {Springer International Publishing},
  address = {Cham},
  doi = {10.1007/978-3-030-04306-3},
  urldate = {2025-05-08},
  copyright = {http://www.springer.com/tdm},
  isbn = {978-3-030-04305-6 978-3-030-04306-3},
  langid = {english}
}

@incollection{pottsFastFourierTransforms2001,
  title = {Fast {{Fourier Transforms}} for {{Nonequispaced Data}}: {{A Tutorial}}},
  shorttitle = {Fast {{Fourier Transforms}} for {{Nonequispaced Data}}},
  booktitle = {Modern {{Sampling Theory}}: {{Mathematics}} and {{Applications}}},
  author = {Potts, Daniel and Steidl, Gabriele and Tasche, Manfred},
  editor = {Benedetto, John J. and Ferreira, Paulo J. S. G.},
  year = 2001,
  pages = {247--270},
  publisher = {Birkh\"auser},
  address = {Boston, MA},
  doi = {10.1007/978-1-4612-0143-4_12},
  urldate = {2025-04-28},
  isbn = {978-1-4612-0143-4},
  langid = {english}
}

@book{schoenbergCardinalSplineInterpolation1973,
  title = {Cardinal {{Spline Interpolation}}},
  author = {Schoenberg, I. J.},
  year = 1973,

  publisher = {{Society for Industrial and Applied Mathematics}},
  doi = {10.1137/1.9781611970555},
  urldate = {2024-12-15},
  isbn = {978-0-89871-009-0 978-1-61197-055-5},
  langid = {english}
}

@article{shamshirgarFastEwaldSummation2021,
  title = {Fast {{Ewald}} Summation for Electrostatic Potentials with Arbitrary Periodicity},
  author = {Shamshirgar, Davood Saffar and Bagge, Joar and Tornberg, Anna-Karin},
  year = 2021,

  journal = {The Journal of Chemical Physics},
  volume = {154},
  number = {16},
  eprint = {1712.04732},
  primaryclass = {cs, math},
  pages = {164109},
  issn = {0021-9606, 1089-7690},
  doi = {10.1063/5.0044895},
  urldate = {2024-10-05},
  archiveprefix = {arXiv},
  langid = {english}
}

@article{slepianCommentsFourierAnalysis1983,
  title = {Some Comments on {{Fourier}} Analysis, Uncertainty and Modeling},
  author = {Slepian, David},
  year = 1983,

  journal = {SIAM Review},
  volume = {25},
  number = {3},
  pages = {379--393},
  publisher = {{Society for Industrial and Applied Mathematics}},
  issn = {0036-1445},
  doi = {10.1137/1025078},
  urldate = {2025-03-23}
}

@article{slepianProlateSpheroidalWave1961,
  title = {Prolate {{Spheroidal Wave Functions}}, {{Fourier Analysis}} and {{Uncertainty}} - {{I}}},
  author = {Slepian, D. and Pollak, H. O.},
  year = 1961,

  journal = {Bell System Technical Journal},
  volume = {40},
  number = {1},
  pages = {43--63},
  issn = {00058580},
  doi = {10.1002/j.1538-7305.1961.tb03976.x},
  urldate = {2024-12-28},
  langid = {english}
}

@article{steidlNoteFastFourier1998,
  title = {A Note on Fast {{Fourier}} Transforms for Nonequispaced Grids},
  author = {Steidl, Gabriele},
  year = 1998,

  journal = {Advances in Computational Mathematics},
  volume = {9},
  number = {3},
  pages = {337--352},
  issn = {1572-9044},
  doi = {10.1023/A:1018901926283},
  urldate = {2025-04-28},
  langid = {english}
}

@book{szegoORTHOGONALPOLYNOMIALS,
  title={Orthogonal polynomials},
  author={Szegö, Gabor},
  volume={23},
  year={1939},
  publisher={American Mathematical Soc.}
}

@article{vanderspoelGROMACSFastFlexible2005,
  title = {{{GROMACS}}: {{Fast}}, Flexible, and Free},
  shorttitle = {{{GROMACS}}},
  author = {Van Der Spoel, David and Lindahl, Erik and Hess, Berk and Groenhof, Gerrit and Mark, Alan E. and Berendsen, Herman J. C.},
  year = 2005,

  journal = {Journal of Computational Chemistry},
  volume = {26},
  number = {16},
  pages = {1701--1718},
  issn = {0192-8651, 1096-987X},
  doi = {10.1002/jcc.20291},
  urldate = {2024-12-29},
  copyright = {http://onlinelibrary.wiley.com/termsAndConditions\#vor},
  langid = {english}
}

\appendix

\section{PSWF asymptotics and approximations}\label{sec:PSWF_appr}

In this section, we summarize the asymptotic behavior of $\psi_0^c$, along with corresponding approximations. These results are used to derive the closed-form approximations in Section~\ref{sec:error_analysis}.

Define
\begin{align}
f(s):=
\begin{cases}
\psi_0^c(s),& s\in[-1,1],\\
0,&\text{otherwise}.\\
\end{cases}
\end{align}
Considering the Fourier tail behavior, the PSWF scaling
(see Theorem~\ref{thm:bandlimited}) yields
\begin{align}
1-\mu_0(c)
  = 2\pi \int_{\R\setminus [-c,c]} |\hat{f}(\omega)|^2\,d\omega
  = 2\pi \|\hat{f}\|^2_{L^2(\R\setminus [-c,c])},
\end{align}
where $\mu_0(c) := c/(2\pi)\,\lambda_0^2$ denotes the first concentration eigenvalue associated
with $\psi_0^c$ and where $\|f\|_{L^2(-1,1)} = \|\psi_0^c\|_{L^2(-1,1)} = 1$.
Fuchs~\cite{fuchsEigenvaluesIntegralEquation1964} proved the asymptotic
behaviour
\begin{align}\label{eq:fuchs}
1-\mu_0(c) \sim 4\sqrt{\pi c}\,e^{-2c} \qquad \text{as } c\to\infty.
\end{align}
Taking the square root yields
\begin{align}\label{eq:fuchs2}
\|\hat{f}\|_{L^2(\R\setminus [-c,c])}
   \sim C_{\mathrm{F}}\, c^{1/4} e^{-c},
\end{align}
where the constant $C_{\mathrm{F}}>0$ depends on the chosen
Fourier-transform convention.

For large $c$, the first PSWF eigenvalue of the integral eigenvalue equation \eqref{eq:eigenvalueproblem} satisfies \cite{osipovProlateSpheroidalWave2013}
\begin{align}\label{eq:lambda}
\lambda_0(c) \sim \sqrt{\frac{2\pi}{c}}.
\end{align}
The Fourier--eigenfunction identity at the end-points
(Lemma~\ref{lem:PSWFrelation})
\begin{align}\label{eq:identity}
\hat{f}(c)=\lambda_0(c)\,\psi_0^c(1)
\end{align}
shows that the decay of $\hat{f}(c)$ determines the decay of
$\psi_0^c(1)$. Using \eqref{eq:lambda} and \eqref{eq:identity} gives
\begin{align}\label{eq:10}
\frac{\psi_0^c(1)}{\psi_0^c(0)}
    \sim \sqrt{\frac{c}{2\pi}}\,
          \frac{\hat{f}(c)}{\psi_0^c(0)}.
\end{align}

Motivated by the strong exponential localization of $\hat{f}$
near $|\omega|=c$, we assume that the pointwise value $\hat{f}(c)$
is of the same order as the tail norm, and write
\begin{align}
\hat{f}(c)
\approx \|\hat{f}\|_{L^2(\R\setminus [-c,c])}.
\end{align}
Furthermore, it is well known from WKB analysis of the PSWF equation
(see Dunster~\cite{dunsterUniformAsymptoticExpansions1986}) that
\begin{align}\label{eq:asymp0}
\psi_0^c(0) \sim C_0 \,c^{1/4}.
\end{align}

Based on the relations above, we fit the following expressions numerically
in the interval $7 \le c \le 35$:
\begin{align}
    \lambda_0 &\approx \sqrt{\frac{2\pi}{c}},\label{eq:asymp_lambda}\\
    \psi_0^c(0) &\approx A_0\, c^{1/4},\label{eq:asymp1}\\
    \psi_0^c(1) &\approx A_1\, c^{3/4}e^{-c},\label{eq:asymp2}\\
    \frac{\psi_0^c(1)}{\psi_0^c(0)} &\approx A_2\, c^{1/2}e^{-c},\label{eq:asymp3}.
\end{align}
Here, \eqref{eq:asymp2} is obtained from \eqref{eq:10} together with
\eqref{eq:fuchs2} and \eqref{eq:asymp0}, and \eqref{eq:asymp3} follows
by dividing \eqref{eq:asymp2} by \eqref{eq:asymp1}.

The results are presented in Table~\ref{tab:curvefits} and
Figure~\ref{fig:curvefit}. For all expressions, it is possible to obtain
a fit with maximum relative errors of about $1\%$. If the lower bound of
the fitting interval is increased above $c=7$, the fit typically improves
further.

The upper limit $c\lesssim 35$ is chosen because the numerical evaluation
of $\psi_0^c$ approaches machine precision beyond this point.

\begin{figure}[tbp]
    \centering
    \includegraphics[trim={0 0 5.5cm 0},clip,width=11cm]{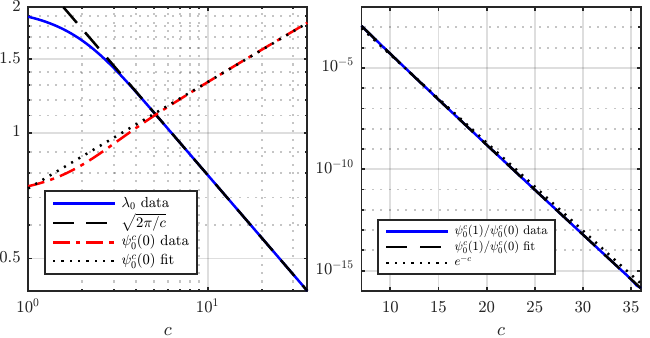}
    \caption{Curve-fitted approximations of key PSWF quantities.}
    \label{fig:curvefit}
\end{figure}

\begin{table}[tbp]
\centering
\caption{Curve-fitted models together with fit quality for $7\leq c\leq 35$.}
\label{tab:curvefits}
\small
\begin{tabular}{lll}
\toprule
Quantity & Model & Max rel.\ error\\
\midrule
$\lambda_0(c)$ & $\sqrt{2\pi/c}$ & $<0.1\%$ at $c=7$\\[5pt]
$\psi_0^c(0)$  & $0.736\,c^{0.2548}$ & $\approx 0.4\%$ at $c=7$\\[5pt]
$\psi_0^c(1)$  & $2.540\,c^{3/4}e^{-c}$ & $\approx 1.6\%$ at $c=7$\\[5pt]
$\psi_0^c(1)/\psi_0^c(0)$
& $3.424\, c^{1/2}\,e^{-c}$ & $\approx 1.5\%$ at $c=35$\\[5pt]
\bottomrule
\end{tabular}
\end{table}

\section{Numerical evaluation of $\psi_0^c$ and $(\psi_0^c)'$}
\label{app:pswf}

The numerical evaluation of $\psi_0^c$ we propose follows the algorithmic
framework of Osipov and Rokhlin~\cite{osipovEvaluationProlateSpheroidal2014}.
A short summary is given below.

\subsection{Legendre expansion and eigenproblem}
The function $\psi_0^c$ is represented as
\begin{align}\label{eq:PSWFLegendre}
\psi_0^c(x) = \sum_{k=0}^{K} a_{2k} P_{2k}(x),
\end{align}
where only even Legendre polynomials appear due to symmetry.
Here $P_k$ denotes the standard (unnormalized) Legendre polynomial; 
the normalized basis $\bar P_k(x)=\sqrt{k+\tfrac12}\,P_k(x)$ used
in~\cite{osipovEvaluationProlateSpheroidal2014} leads to an equivalent
eigenproblem with rescaled coefficients.
Substitution into the differential equation \eqref{eq:ode} 
leads to a three-term recurrence for the coefficients $\{a_{2k}\}$.
After truncation at order $K$, a finite-dimensional symmetric
tridiagonal eigenproblem is obtained. Its smallest eigenvalue
corresponds to $\chi_0$ and the associated eigenvector provides
the coefficients $a_{2k}$ \cite[Theorem~10]{osipovEvaluationProlateSpheroidal2014}.
\subsection{Computation of eigenpairs}
The tridiagonal structure permits efficient and accurate computation
of the eigenvalue $\chi_0$ and the coefficient vector $\{a_{2k}\}$.
In practice, Sturm bisection followed by inverse iteration is applied,
as recommended in \cite[Sec.~3.4]{osipovEvaluationProlateSpheroidal2014}.
This yields rapidly convergent approximations to $\psi_0^c$ with
complexity $\mathcal{O}(K^2)$.

\subsection{Evaluation of $\psi_0^c$ and $\lambda_0$}
Once the Legendre coefficients are available, $\psi_0^c$ is evaluated
directly from its expansion. The eigenvalue $\lambda_0(c)$ of the
integral operator is then obtained from
\begin{align}
\lambda_0 = \frac{1}{\psi_0^c(0)} \int_{-1}^1 \psi_0^c(s)\,ds,
\end{align}
using high-order Gauss–Legendre quadrature
\cite[Sec.~3.2]{osipovEvaluationProlateSpheroidal2014}.

\subsection{Evaluation of $(\psi_0^c)'$}

Given the Legendre expansion \eqref{eq:PSWFLegendre} the derivative can be computed as
\begin{align}
    (\psi_0^c)'(x) = \sum_{k=0}^{K} a_{2k} P'_{2k}(x).
\end{align}
Another way is to differentiate \eqref{eq:eigenvalueproblem} and evaluate the following integral using a numerical quadrature:
\begin{align}
    (\psi_0^c)'(x) = \frac{i c}{\lambda_0} \int_{-1}^1 t\,\psi_0^c(t)\,e^{icxt}\,\d t.
\end{align}

\subsection{Truncation order}
The required number of terms $K$ in the expansion grows linearly with
the bandlimit $c$. Empirically, $K \approx 1.2c$ suffices to achieve
$10^{-12}$ accuracy, in agreement with the decay estimates for the
Legendre coefficients in
\cite[Remark~9]{osipovEvaluationProlateSpheroidal2014}.

\section{Proofs of Lemmas \ref{lem:chi0} and \ref{lem:tailbnd}}\label{sec:tailbnd}

\begin{proof}[Proof of Lemma~\ref{lem:chi0}]
The Rayleigh quotient associated with \eqref{eq:ode} is 
\begin{align}
\chi_0 = \min_{\varphi \ne 0} 
    \frac{\int_{-1}^{1} (1-s^2)|\varphi'(s)|^2 \, \d s
    +c^2\int_{-1}^1 s^2|\varphi(s)|^2\,\d s}
    {\int_{-1}^{1} |\varphi(s)|^2 \,\d s}
    .
\end{align}
Then, using the test function $\varphi(s)=1$,
\begin{align}
\chi_0 \le
    \frac{c^2\int_{-1}^1 s^2\,\d s}{\int_{-1}^{1} 1 \,\d s}
    = \frac{c^2}{3}.
\end{align}
\end{proof}

\begin{proof}[Proof of Lemma~\ref{lem:tailbnd}]
Let $u(s):= s \psi_0^c(s)$. Substituting $\psi_0^c = u/s$ into \eqref{eq:ode}, we obtain
\begin{align}\label{eq:ueqn}
    p u'' + p' u' + q u = 0,
\end{align}
where
\begin{align}
    p(s) := \frac{s^2 - 1}{s^2}, 
    \qquad 
    q(s) := \frac{s^2 (c^2 s^2 - \chi_0) - 2}{s^4}
    = c^2-\chi_0 s^{-2}-2s^{-4}.
\end{align}
For $s>1$, define the Sonin-type functional (see, e.g.,~\cite[p.~166]{szegoORTHOGONALPOLYNOMIALS})
\begin{align}
    F(s) := u(s)^2 + \frac{p(s)}{q(s)} (u'(s))^2.
\end{align}
Since $p(s)>0$ for $s>1$, it remains to verify that $q(s)>0$ on $(1,\infty)$.

Using Lemma~\ref{lem:chi0}, $\chi_0\le c^2/3$, and assuming $c\ge \sqrt{3}$, we obtain for $s>1$,
\begin{align}
q(s)
= c^2-\chi_0 s^{-2}-2s^{-4}
> c^2-\chi_0-2
\ge \tfrac{2}{3}c^2-2
\ge 0.
\end{align}
Hence, $q(s)>0$ for all $s>1$, and therefore $F$ is well defined and nonnegative on $(1,\infty)$.

Using \eqref{eq:ueqn}, in particular $p u'' = - p' u' - q u$, differentiation yields
\begin{align}
    F'(s) = - \frac{(pq)'(s)}{q(s)^2} (u'(s))^2.
\end{align}
Moreover,
\begin{align}
(pq)'(s)
=
\frac{2(c^2+\chi_0)s^4-4(\chi_0-2)s^2-12}{s^7}.
\end{align}
For $s>1$, the numerator is bounded below by
\begin{align}
2(c^2+\chi_0)-4(\chi_0-2)-12
= 2(c^2-\chi_0-2)
\ge 2\left(\frac{2}{3}c^2-2\right)
\ge 0,
\end{align}
where we again used Lemma~\ref{lem:chi0} and $c\ge \sqrt{3}$. Hence, $(pq)'(s)\ge 0$ for all $s>1$, and therefore $F'(s)\le 0$ on $(1,\infty)$.

Thus, $F$ is nonincreasing on $(1,\infty)$. Therefore,
\begin{align}
    u(s)^2 \le F(s) \le \lim_{t\to 1^+} F(t)=u(1)^2=\bigl(\psi_0^c(1)\bigr)^2.
\end{align}
It follows that
\begin{align}
|\psi_0^{c}(s)|\le \frac{\psi_0^c(1)}{s}.
\end{align}
\end{proof}

\section{Auxiliary standard results}
\label{sec:standard}

\subsection{Radial extension (Fourier-side definition)}
\label{app:radial-extension}

Let $f:\R\to\R$ be even and belong to $L^1(\R)$, with Fourier transform
$\widehat{f}$. Define the radially symmetric function
$\eta:\R^3\to\R$ by
\begin{align}
\widehat{\eta}(\bm{\omega}) := \widehat{f}\big(|\bm{\omega}|\big),
\qquad \bm{\omega}\in\R^3.
\end{align}
Equivalently, in physical space,
\begin{align}
\eta(\bx)
\;=\;
\frac{1}{2\pi^2}
\int_{0}^{\infty}
\widehat{f}(\rho)\,
\frac{\sin(\rho |\bx|)}{\rho |\bx|}\,
\rho^2\,d\rho.
\end{align}

\begin{remark}
In general, $\eta(\bx)\neq f(|\bx|)$; the identification is through the
three-dimensional radial (spherical Bessel) integral above, not by
pointwise substitution.
\end{remark}

\subsection{Derivation of \eqref{eq:fcoeffs}}
\label{sec:fcoeffs}

For completeness, we provide the standard calculation leading to
\eqref{eq:fcoeffs}. Starting from the definition of the Fourier
coefficients of the periodized function $\tilde f$,
\begin{align}
c_{\bm{k}}(\tilde{f})
=
\frac{1}{V}
\int_{\Omega}
\tilde{f}(\bm{x})\,
e^{-\frac{2\pi i}{L}\bm{k}\cdot\bm{x}}
\,d\bm{x},
\end{align}
and substituting the definition of $\tilde f$, we obtain
\begin{align}
c_{\bm{k}}(\tilde{f})
&=
\frac{1}{V}
\int_{\Omega}
\sum_{\bm{r}\in\mathbb{Z}^3}
f(\bm{x}+L\bm{r})\,
e^{-\frac{2\pi i}{L}\bm{k}\cdot\bm{x}}
\,d\bm{x}
\nonumber\\
&=
\frac{1}{V}
\sum_{\bm{r}\in\mathbb{Z}^3}
\int_{\Omega+L\bm{r}}
f(\bm{x})\,
e^{-\frac{2\pi i}{L}\bm{k}\cdot(\bm{x}-L\bm{r})}
\,d\bm{x}
\nonumber\\
&=
\frac{1}{V}
\int_{\R^3}
f(\bm{x})\,
e^{-\frac{2\pi i}{L}\bm{k}\cdot\bm{x}}
\,d\bm{x},
\qquad
\text{since } e^{2\pi i\,\bm{k}\cdot\bm{r}}=1,
\nonumber\\
&=
\frac{1}{V}\,\widehat{f}(\bm{k}).
\end{align}

\subsection{A lattice-to-continuum estimate}
\label{app:cubecover}

We recall a standard lattice covering argument that bounds radial lattice
sums by corresponding integrals. This estimate is used in the proof of
Proposition~\ref{prop:PSWF_split_errors} to pass from discrete sums to radial
integrals.

\begin{lemma}[Lattice covering lemma in $\R^3$]
\label{lem:cubecover}
Let $f:[0,\infty)\to[0,\infty)$ be nonincreasing. For $K>\sqrt{3}/2$,
\begin{align}
\sum_{|\bm{k}|>K} f(|\bm{k}|)
\;\le\;
\int_{|\bx|>K-\sqrt{3}/2}
f(|\bx|)\,d\bx.
\end{align}
\end{lemma}

\begin{proof}
Partition $\R^3$ into cubes
$Q_{\bm{k}}=\bm{k}+[-\tfrac12,\tfrac12)^3$.
For any $\bx\in Q_{\bm{k}}$ one has
$||\bx|-|\bm{k}||\le\sqrt{3}/2$.
Monotonicity of $f$ then yields the stated bound.
\end{proof}

\end{document}